\documentclass{siamltex}		
\usepackage[utf8]{inputenc}
\usepackage{graphicx,multirow}  
\usepackage{verbatim,float} 
\usepackage{todonotes}
\usepackage{amssymb,amsmath}
\usepackage[T1]{fontenc}
\usepackage{graphicx,amsmath,amsfonts,amssymb,subfigure}
\usepackage{color}
\usepackage{hyperref}
\usepackage[ruled,vlined,linesnumbered]{algorithm2e}

\def\cmm{{\cal M}}
\def\cma{{\cal A}}

\def\cU{{\cal U}}
\def\tla{{\widetilde \lambda}}
\def\tlu{{\widetilde u}}
\def\CC{\mathbb{C}}
\def\up#1{^{({#1})}} %
\def\nref#1{(\ref{#1})}
\def\inv{^{-1}}%
\def\eps{\epsilon}%
\def\tendsto{\rightarrow}
\newcommand{\ca}[1]{\mathcal{#1}}
\newcommand{\eq}[1]{\begin{equation}\label{#1}}
\newcommand{\en}{\end{equation}}
\usepackage{multirow}

\graphicspath{{./FIGS/}}

\usepackage{verbatim}

\newcommand{\xbmatrix}[1]{\begin{bmatrix} #1 \end{bmatrix}}

\title{A rational approximation method for the nonlinear eigenvalue problem} 

\author{Yousef Saad \thanks{University of Minnesota, Department of Computer Science 
\& Engineering, 4-192 Keller Hall, 200 Union Street SE, Minneapolis, MN 55455,  USA.
  Work supported by NSF grant 1812695.
  e-mail: \texttt{saad@umn.edu}}
\and Mohamed El-Guide \thanks{Mohammed VI Polytechnic University, Centre for Behavioral Economics and Decision Making (CBED), Lot 660, Hay Moulay Rachid, Ben Guerir 43150, Green City, Morocco. Work supported by NSF grant 1812695  e-mail: \texttt{mohamed.elguide@um6p.ma}}
  \and Agnieszka Mi\k{e}dlar
  \thanks{University of Kansas, Department of Mathematics, 405 Snow Hall, 1460 Jayhawk Blvd., Lawrence, KS 66045-7594, USA. Work supported by NSF grant  1812927. e-mail:
  \texttt{amiedlar@ku.edu}}
  }
\begin{document} 

\maketitle 
\begin{abstract}
  This paper presents a method for computing eigenvalues and
  eigenvectors for some types of nonlinear eigenvalue problems.  The
  main idea is to approximate the functions involved in the eigenvalue
  problem by rational functions and then apply a form of
  linearization.  Eigenpairs of the expanded form of this
  linearization are not extracted directly. Instead, its structure is
  exploited to develop a scheme that allows to extract all eigenvalues
  in a certain region of the complex plane by solving an eigenvalue
  problem of much smaller dimension. Because of its simple
  implementation and the ability to work efficiently in large
  dimensions, the presented method is appealing when 
  solving challenging engineering problems. A few theoretical results
  are established to explain why the new approach works and
  numerical experiments are presented to validate the proposed
  algorithm.
\end{abstract}



\section{Background and introduction}
Consider a non-empty open set $\Omega \subset \CC$ and a matrix-valued
function $T(z): \Omega \rightarrow \CC^{n \times n}$
that is analytic on
$\Omega$, i.e., each component of $T(z)$ is an analytic function of
$z$. The nonlinear eigenvalue problem (NLEVP) consists of finding
$\lambda \ \in \Omega $ and a nonzero vector $u \ \in \ \CC^n$ such
  that:
  \eq{eq:nlin0}{}
  T(\lambda)u = 0.
  \en
  We call $\lambda$ an
eigenvalue of $T(\cdot)$ and $u$ the associated eigenvector.  Problems
of this type arise in numerous applications, including in the analysis
of vibration of rails under excitation from fast
trains~\cite{Kli03,ApeMW04}, in the optimization of acoustic emissions
of high speed trains~\cite{MacMMM06}, in electronic structure
calculations of quantum dots~\cite{Hwa04,Vos06}, and in the study of
photonic resonators at the nanoscale~\cite{Bha09,Val03}.  As a result,
these problems have been extensively studied in the
literature~\cite{TisM01,MehV04,Vos14,MacMT15,GueT17} and a plethora of
specialized methods were developed for different types of structures
of $T(z)$, see e.g.~\cite{TisM01} for quadratic
or~\cite{MacMT15,MacMMM06a} for polynomial eigenvalue problems.

The first type of nonlinear eigenvalue problems that were studied were
polynomial eigenvalue problems (PEPs) in which $T(z)$ is a
polynomial of small degree in $z$, with matrix coefficients.
These problems can be easily converted
into an equivalent (same eigenvalues) generalized eigenvalue problem
via \textit{linearization}.  The resulting, larger problem, is then
solved by standard techniques~\cite{GolV13}.  Linearization-based
approaches for solving polynomial eigenvalue problems have been
extensively studied, see, e.g.,~\cite{GohLR09,MehV04}.

This work focuses on computing the eigenvalues of
general nonlinear eigenvalue problems, i.e., those for which
$T(z)$ is not a polynomial. The goal is to compute  all eigenvalues located
inside a closed contour $\Gamma$ of the complex plane.
\nocite{AsaSTIK09,BeyEK11,Bey12,Eff13,YokS13,VBar16,VBarK16,GavMP18}
\nocite{SuB11,JarMM12,GueVBMM14,VanB15,DopGP18,LiePVM18}
State-of-the-art numerical methods for general nonlinear eigenvalue
problems include Newton-type methods, e.g.,
\cite{Kub70,Lan66,Ung50,Ruh73,Neu85,Lan61,Sch08,SleBFV96,BetV04,Kre09},
techniques based on contour
integration~\cite{AsaSTIK09,BeyEK11,Bey12,Eff13b,YokS13,GavMP18}, and
methods based on polynomial and rational approximations of
$T(z)$, e.g., ~\cite{MehV04, MacMMM06}.

The method proposed in this paper belongs to the last of the three categories listed
above and it relies specifically on rational approximations obtained from
the Cauchy integral formula.  Among  methods that exploit contour integrals, 
Beyn's method~\cite{Bey12} figures prominently and is currently the best known.
Given a matrix $Q$ which contains a set of $k$ vectors, e.g., chosen randomly,
Beyn's method exploits the relation between the
following matrices, called \textit{$p$-th (order) moments},
\begin{equation}
\label{eq:moments_Sp}
S_p = \frac{1}{2\pi \imath} \int\limits_{\Gamma}z^pT(z)^{-1}Q \ dz,\quad p = 0,1,2,\ldots
\end{equation}
By Keldysh's theorem, under certain conditions, $T(z)\inv$ can be written locally as
$T(z)\inv = U (z I - \Lambda)\inv W^H + R(z)$,
where $R(z)$ is analytic and therefore:
\begin{eqnarray}
\label{eq:zero_first_moments}
S_0 & = & \frac{1}{2\pi \imath} \int\limits_{\Gamma} T(z)^{-1}Q \ dz = UW^*Q \ \in \CC^{n \times k}, \nonumber \\
S_1 & = & \frac{1}{2\pi \imath} \int\limits_{\Gamma} zT(z)^{-1}Q \ dz = U\Lambda W^*Q \ \in \CC^{n \times k}.
\end{eqnarray}
The idea then is to exploit the relation between the two matrices on the right-hand sides
of the above equations, to extract $\Lambda$. For this purpose, the algorithm relies on the singular value decomposition (SVD).


Different interpretations of the method just sketched have been
exploited.  Thus, the article~\cite{VanB15} made a link between Beyn's
method and \emph{rational filters} which are exploited to ``filter'' (extract) the
approximate invariant subspace corresponding to the eigenvalues
$\lambda_i, i = 1, \ldots, n_\Gamma,$ located inside $\Gamma$. This led 
to a study of general filter functions of the form
\begin{equation}
\label{eq:filter}
\sum\limits_{j = 0}^{N - 1} \frac{\omega_j z_j^p}{z_j - z},
\end{equation}
for solving NLEVPs in \cite{VanB15} and  in ~\cite{VBarK16}.  The
filtering nature of contour integral methods is also explored in the
Nonlinear FEAST algorithm~\cite{GavMP18}. In this context,
we also mention the interesting work
by Embree et al.~\cite{EmbGBG19} who make the connection with \emph{transfer
  functions} of  dynamics in system theory.

The method proposed in this paper takes a general NLEVP, approximates it with
a rational eigenvalue problem and then solves this rational problem by a form of
linearization. This general approach is not new and the following
is a short, albeit incomplete,  overview of this class of techniques.
In the  rational Krylov-based approach known as
\textit{Newton rational Krylov} proposed in~\cite{VanBMM13}, the 
matrix-valued function $T(z)$ is first approximated by a
low-degree Hermite interpolating polynomial in Newton form.
The resulting generalized eigenvalue
problem is then solved by a rational Krylov method, with the
interpolations points taken as shifts. This leads to a   flexible
approach that can easily incorporate information from the most recent
iterations, take advantage of the underlying structure of the problem
and simultaneously find several eigenvalues of interest. To improve
convergence, the Newton rational Krylov 
method was generalized from polynomial to linear rational
interpolation, resulting in an algorithm called \textit{fully rational
  Krylov method} for NLEVPs, commonly refereed as \textit{NLEIGS},
see~\cite{GueVBMM14}. Of
particular interest is the dynamic NLEIGS variant which utilizes the
rational Newton expansion and the companion-like strong linearizations
to dynamically add interpolation nodes and poles to extend the matrix
pencils, and to merge the construction of the rational approximation with the
application of the rational Krylov method.

When solving large-scale nonlinear eigenvalue problems, it is
essential to exploit the structure of the linearized problem in order
to overcome the large memory and computational costs. This is the
primary goal of \textit{compact rational Krylov} (CORK) framework
proposed in~\cite{VanBMM15}.  In this general framework it is assumed
that the approximation of $T(z)$, whether polynomial or
rational, is put in the form $\widetilde T(z) =
\sum_{i=0}^{d-1}(A_i - \lambda B_i)f_i(z), $ with the
scalar functions $f_i(z)$ satisfying a linear recurrence
relation $ M_d \textbf{f}_d = z N_d \textbf{f}_d(z), $
where $\textbf{f}_d = \big[ f_0(z), f_1(z), \ldots,
  f_{d-1}(z) \big]^T$ and $M_d, N_d \in \CC^{(d-1) \times
  d}$.
Then the associated companion linearization, often referred as
\textit{CORK linearization}, turns out to have a particularly
interesting structure, namely: $\textbf{L}(z) = \textbf{A} -
z \textbf{B}$, where
\begin{equation}
\label{eq:CORKpencil}
\textbf{A} = \left[ \begin{array}{cccc}
              A_0 & A_1 & \cdots & A_{d-1}\\
              \hline
              \multicolumn{4}{c}{M_d \otimes I_n}
             \end{array} \right], \qquad
\textbf{B} = \left[ \begin{array}{cccc}
              B_0 & B_1 & \cdots & B_{d-1}\\
              \hline
              \multicolumn{4}{c}{N_d \otimes I_n}
             \end{array} \right].
\end{equation}
If $\widetilde T(z)$ is a matrix polynomial, the pencil
\eqref{eq:CORKpencil} is the classical companion-like
linearization~\cite{GohLR09}.
The Kronecker structure of the CORK pencil~\eqref{eq:CORKpencil}
allows to efficiently solve the associated generalized eigenvalue
problem using a rational Krylov method~\cite{Ruh84,Ruh98}. With a
compact representation $v = (I\otimes Q)u$ (Arnoldi
decomposition~\cite{SuZB08}) of the right Krylov vectors $v$, with $Q$
having orthonormal columns and $u$ being of much smaller dimension
than $v$, CORK can be characterized as a two-step procedure similar to
the two-level orthogonal Arnoldi (TOAR)~\cite{LuSB16}. The
orthogonalization step involving vectors of original problem size $n$
followed by a standard rational Krylov step on a projected matrix
pencil, significantly lower the overall memory and computational costs
of the CORK algorithm. Moreover, both the implicit restarting
procedure and utilizing low-rank structure of coefficient matrices
$A_i$ and $B_i$  make CORK method highly
competitive when it comes to solving efficiently and reliably
challenging nonlinear eigenvalue
problems~\cite{VanBONYBGKLNX18}.
Extensions and refinements of the CORK framework were proposed in
\cite{LieMT18,RobVVD17,DopGP18,SuB11,LieMT18,LiePVM18,NakST18}.
 
The method we propose is in the same family as those described above but there are distinctions. First, we rely entirely on the
Cauchy integral formula to approximate $T(z)$ directly. As will be seen
in Section~\ref{sec:alg},  the resulting
approximation is a rational function which includes simple terms of the form
$B_i /(z-\sigma_i)$. 
As it was already made clear above, an essential part of the methods for solving
rational eigenvalue problems, is exploiting a good linearization.
A large volume of work has been devoted to linearizations both from a
 practical and from a theoretical viewpoint
 \cite{AmiCL09,AntV04,DopPVD19,DopLPVD18,FasS18,Fie03,GohLR09,HigMMT06,
   MacMT15,Mac06,Mac13,MacMMM06,DeTD08,DeTDM10,DeTDM14}.
  We propose a simple and natural linearization technique which lends itself to
  efficient calculations.

The method we propose in this paper exploits a projection technique that employs vectors
of length $n$, the size of the original problem.  A slight
modification of the method was recently used to solve 
a rather challenging  nonlinear eigenvalue  problem that arises from applying
the   Boundary Element Method    in acoustics 
\cite{ElGuide19}. The experiments proposed at the end of the paper are 
presented primarily for illustrating certain characteristics of the method
including its  versatility and ease of use.

\section{A rational approximation approach for NLEVPs} \label{sec:alg}

Following Kressner \cite{Kre09}, we  limit ourselves to problems in which:
\eq{eq:Tofzm1}
T(z) =  f_0 (z) A_0 + f_1(z) A_1 + f_2(z) A_2 +\ldots + f_p(z) A_p,
\en
with holomorphic  functions $f_0,\ldots, f_p: \Omega  \rightarrow \CC$
and constant coefficient matrices $A_0, \ldots, A_p$.
In what follows, we will call $\Gamma $ the boundary of $\Omega$. 
Since $T\in  H(\Omega, \CC)$,  it can  always be  written in  the form
\eqref{eq:Tofzm1} with  at most $p  = n^2$ terms. Note also that
this representation  is not unique.  Furthermore, it is very  common in
practice to have $f_0(z)= 1$ and $f_1(z) = z$, so instead of the above
we will assume  the form: \eq{eq:Tofz} T(z)  = - B_0 + z  A_0 + f_1(z)
A_1 + \ldots + f_p(z) A_p .  \en
As it turns out many of the nonlinear
eigenvalue problems  encountered in applications
are  set in this  form.   We are
interested in  all eigenvalues  that are  located in  a region  of the
complex plane enclosed by curve $\Gamma$. 

The main assumption we make is  that each of the holomorphic functions
$f_j: \Omega \rightarrow \CC$ in representation \nref{eq:Tofz} is well
approximated by a rational function of the form:
\eq{eq:fj}
f_j(z)  \approx \sum_{i=1}^m \frac{\alpha_{ij} }{z - \sigma_i}.
\en
Note that the set of $m$ poles $\sigma_i$'s is the same for all of the functions.
This setting comes from a Cauchy integral representation of each function
inside a region limited by a contour $\Gamma$:
\eq{eq:fj0}
f_j(z)  = -\frac{1}{2 i \pi} \int_\Gamma \frac{f_j(t) }{z-t}\  dt, \quad
\quad z \in \Omega. 
\en
Using numerical  quadrature, \eqref{eq:fj0}  is then  approximated into
\nref{eq:fj}, where the $\sigma_i$'s  are quadrature points located on
the contour $\Gamma$. Substituting \nref{eq:fj} into \nref{eq:Tofz} yields the following approximation $\widetilde T(z)$ of $T(z)$:
\begin{align}  
\widetilde T(z) &= -B_0+ z A_0 + 
\sum_{j=1}^p \sum_{i=1}^m
                  \frac{\alpha_{ij} }{z - \sigma_i} A_j =
                  -B_0+ z A_0 + \sum_{i=1}^m \
\frac{ \sum_{j=0}^p \alpha_{ij} A_j }{z - \sigma_i}  \nonumber \\
        &\equiv  \ -B_0+ z A_0 + 
\sum_{i=1}^m \
\frac{ B_i }{z - \sigma_i},  \label{eq:Tofz1}
\end{align}
where we have set
\eq{eq:Bi}
B_i = \sum_{j=0}^p \alpha_{ij} A_j, \qquad i=1, \ldots, m.
\en
Given   the  approximation   $\widetilde   T(z)$  of   $T(z)$ shown   in
\nref{eq:Tofz1},  the problem  we  need  to solve  can  be written  as
follows:
\eq{eq:NLR} \Big(-B_0 + \lambda A_0 + \sum_{i=1}^m \
\frac{ B_i }{\lambda  - \sigma_i} \Big) u = 0 \ .
\en
We will often refer to this problem
as  a \emph{surrogate}  for problem~\nref{eq:nlin0}. It will  be seen
that  if each  of the  functions $f_j$  is well  approximated then this
surrogate problem will provide 
good approximations to eigenvalues of $T(z)$ located inside the
contour but that are not close to the poles. 
  
For a given complex number $z$, and a given vector $u$, we define
\eq{eq:vi}
v_i = \frac{ u}{\sigma_i - z} , \qquad i=1,\ldots, m .
\en
Then we can write
\begin{align}
\label{eq:Tzu}
  \widetilde T(z) u & = \Big( -B_0 + zA_0 + \sum_{i=1}^m \frac{ B_i }{z - \sigma_i}\Big) u\\ 
                  & = (-B_0 + zA_0)u -  \sum_{i=1}^m B_iv_i,
\end{align}
which can be expressed in  block form as follows: 
\eq{eq:Block}
\begin{bmatrix} 
(z - \sigma_1)I &                &               &                 &     I          \\
                & (z- \sigma_2)I &               &                 &     I           \\        
                &                & \ddots        &                & \vdots           \\
                &                &         &  (z-\sigma_m)I    & I           \\
                -B_1             & -B_2            & \ldots        &   -B_m          & z A_0 - B_0 
              \end{bmatrix}
               w = 0,
              \quad 
               w = 
              \begin{bmatrix}
                v_1 \\ v_2 \\ \vdots \\ v_m \\ u
              \end{bmatrix} \ .               
              \en
Since \eqref{eq:Block} is of the form $(z \cmm  - \cma ) w = 0$, solutions of the  surrogate eigenvalue problem $\widetilde T(\lambda)  u = 0$ can be obtained by
solving the linear eigenvalue problem 
\eq{eq:sc}
{\cal A}  w = \lambda {\cal M} w,
\en
with
\eq{eq:sc1}
\cmm = 
\begin{bmatrix} 
    I  &         &            &                &                   \\
       &       I &            &                &               \\        
       &         & \ddots     &                &            \\
       &         &            &   \ddots       &            \\
       &         &            &                &   A_0
     \end{bmatrix}, \qquad      
\cma = 
\begin{bmatrix} 
  \sigma_1 I &                &               &                 &      -I          \\
                &   \sigma_2 I &               &                 &      -I          \\        
                &                & \ddots        &                & \vdots           \\
                &                &         &    \sigma_m I            &   -I           \\
               B_1              & B_2            & \ldots        &   B_m          &  B_0 
                
             \end{bmatrix} . 
\en 

Note  that  the  diagonal  block with  the  $\sigma_i$'s  is  of
dimension $(mn) \times  (mn)$ and the matrices $\cmm$,  and $\cma$ are
each of dimension  $(mn+n) \times  (mn+n)$.  The  above formalism  provides a
basis for developing algorithms  to extract approximate eigenvalues of
the  original  problem  \nref{eq:Tofz}, however, we  will  not
store the matrices $\cma$ and $\cmm$ explicitly.

\subsection{Shift-and-invert on full system} 
Since  we are interested  in interior eigenvalues,  it is
imperative to  exploit a shift-and-invert strategy,  which consists of
replacing  the   solution  of   problem (\ref{eq:sc}) by
\eq{eq:sc2}  {\cal H}  w  =
\frac{1}{\lambda-\sigma} w,\qquad {\cal H}:= \left({\cal  A }-\sigma {\cal M}
\right)^{-1}{\cal M},
\en
where $\sigma$ is a certain shift. 
In the following we will show how to exploit the structure of the linearized
problem (\ref{eq:sc} -- \ref{eq:sc1}) to perform one step of shifted inverse
iteration. This is a basic ingredient which will be utilized in
various ways later.
The shifted inverse iterations require solving linear systems
with a shifted matrix $(\cma  - \sigma \cmm )$ at each step.
To solve such systems, we can exploit a 
standard block LU factorization that takes advantage of the specific patterns of ${\cal A  }$ and ${\cal M  }$.
As a consequence of \eqref{eq:sc1}, all these systems are of the form 
\eq{eq:sc3}
\begin{bmatrix} 
	D & F \\
	B^T & B_0
\end{bmatrix} \begin{bmatrix} x \\ y \end{bmatrix} = \begin{bmatrix} a \\ b \end{bmatrix} ,
\en
where $D$ is  diagonal. Consider the block LU factorization of matrix $\cma$: 
\eq{eq:sc4}
L=\begin{bmatrix} 
	I & 0 \\
	B^T D^{-1} & I
\end{bmatrix}, \qquad 
U=\begin{bmatrix} 
D & F \\
0 & S
\end{bmatrix},
\en
where $S = B_0 - B^T  D^{-1}F$ is the Schur complement of the
block $B_0$. Solving \nref{eq:sc3} requires first solving the system 
$Sy=b-B^T D^{-1}a$ and then substituting $y$ in the first part of \nref{eq:sc3} to obtain
$x = D \inv( a - F y)$.
Next we examine carefully the iterates of the inverse power method
(inverse iteration) or shift-and-invert method to see how they can be
integrated into a projection-type procedure.
We write the iterates obtained from an
inverse iteration procedure as $w\up{k} = \left[ v\up{k} \ ; \ u\up{k} \right] $
where we used {\sc Matlab} notation $[x \ ; \ y ]$ to
denote a vector that consists of $x$ stacked above $y$.

  Each step of the shifted
  inverse power method (inverse iteration) requires solving the linear system
\eq{eq:invpow}
 (\cma - \sigma \cmm)  w\up{k+1} = \ca{M} w\up{k} \quad \mbox{ or } \quad
 (\cma - \sigma \cmm) \xbmatrix{ v\up{k+1} \\ u\up{k+1} } = \xbmatrix{ v\up{k} \\ A_0 u\up{k}}.
\en
The  system \nref{eq:invpow} is of the same form as that in \nref{eq:sc3}
and it can be solved  the same way, resulting in
the following steps~: 
\begin{align}
  u\up{k+1} & = S(\sigma)\inv\big(A_0 u\up{k} - B^T (D-\sigma I)\inv v\up{k}\big),
  \label{eq:wkp1}\\
 v\up{k+1} & = (D-\sigma I) \inv (v\up{k} -F u\up{k+1}) .  \label{eq:vkp1}
\end{align} 
Algorithm~\ref{alg:shinv} shows an implementation of a single step of
this scheme, in which the operations $(D - \sigma I) \inv v\up{k}$ are
translated by scalings on each of the subvectors.


\begin{algorithm}[h!]
\label{alg:shinv}
\SetKwInOut{Input}{Input}\SetKwInOut{Output}{Output}
\Input{$ D, F,B^T \mbox{ and } B_0$ as defined in \eqref{eq:sc3},
  $w\up{k} = [ v\up{k} ; \ u\up{k} ]  $ }
\Output{ $w\up{k+1} = [ v\up{k+1} ; \ u\up{k+1} ] $}
\BlankLine
Compute $b = A_0 u\up{k} - B^T(D -\sigma I)^{-1}v\up{k}
= A_0 u\up{k} - \sum_{i=1}^m (\sigma_i - \sigma)\inv B_i v_i\up{k};  $\\
Solve $S(\sigma) u\up{k+1} = b$, with Schur complement matrix $S(\sigma)$; \\ 
Set $v_i\up{k+1} = [v_i \up{k} - u\up{k+1}]/(\sigma_i - \sigma)$ for $i=1,\ldots, m$; \\
\Return{ $[v\up{k+1}; \ u\up{k+1}]$}
\caption{Single step of shifted inverse iteration}
\end{algorithm}\DecMargin{1em}

The second part of line~1 of the algorithm executes the operation
$B^T (D-\sigma I) \inv $, exploiting
the block structure. Note that the superscripts $k$ correspond to the iteration
number while the subscripts $i$ correspond to the blocks in the vector $v\up{k}$.
Similarly, line~3 unfolds the operation represented by \nref{eq:vkp1} into blocks.
 

In  theory, the  above  single  vector procedure  can now  be applied
in combination with a Krylov subspace method, e.g., the Arnoldi procedure,
to yield a shift-and-invert Arnoldi method applied to the large system
(\ref{eq:sc}--\ref{eq:sc1}). The issue with this approach is that we need to store potentially  many vectors,
 each of  length $(m+1)n$ because  the Arnoldi  procedure
 requires saving all  previous basis vectors in a  given iteration. If
 $m$  is  large,  this will  lead  to  a  big  demand of  memory.
 An alternative  available is the subspace iteration (SI) method which is the key
 ingredient used in the FEAST algorithm~\cite{FEAST}. In contrast with the
 shift-and-invert Arnoldi, SI has the attractive feature of requiring a fixed number of 
 vectors and is known for its robustness.
 Although the basic SI algorithm still has
 the drawback of employing long vectors, we will now discuss a variant to circumvent this issue.

\subsection{Projection method on the reduced system}\label{sec:reduc} 
This  section describes  a projection  method that  works in  $\CC^n$,
i.e., it only requires vectors of length $n$, the size of the original
problem  \eqref{eq:nlin0}.   Let  us consider  the  surrogate  problem
\nref{eq:NLR}. For now, we assume that  we are able to find a subspace
$\cU$, which  contains good approximations to  eigenvectors of problem
\nref{eq:nlin0}, where $T(z)$  is of the form  \nref{eq:Tofz}. For the
sake of  simplicity of the  presentation, we will focus  on orthogonal
projection methods here, noting that generalizations to non-orthogonal
methods are straightforward.

Let $U = [u_1, u_2, \ldots, u_\nu]$ be an orthonormal basis of $\cU$. An approximate eigenvector $\widetilde u$ can be expressed in this basis
as $\widetilde u = U y$, with $y \in \CC^\nu$.  Then, a Rayleigh-Ritz procedure applied to \nref{eq:nlin0} yields a projected problem:
\eq{eq:RR1}
U^H 
\Big( -B_0+ z A_0 + \sum_{i=1}^m \
  \frac{ B_i }{z - \sigma_i} \Big) U y = 0.
\en 
This leads to a nonlinear eigenvalue problem in $\CC^\nu$, namely:
\eq{eq:RR2}
\Big(  - \widetilde B_0+ z \widetilde A_0 + \sum_{i=1}^m \
  \frac{ \widetilde  B_i }{z - \sigma_i} \Big) y = 0,
\en 
in which $\widetilde A_0 = U^H A_0 U$, and
$\widetilde B_i = U^H B_i U, $ for $i=0, 1, \ldots, m$. When $\nu$ is small this can be handled
by solving problem (\ref{eq:sc} -- \ref{eq:sc1}) directly by standard methods,
even if $m$ is fairly large.

The question that still remains to be answered is how to obtain a good
subspace $\cU$ to perform the projection method.  Here, we will rely
once more on the linear form (\ref{eq:sc} -- \ref{eq:sc1}) and the
vectors obtained from a shift-and-invert iteration.  To motivate our
approach of obtaining a good basis $U$, suppose we wish to perform a
single step  of the subspace iteration algorithm applied with
shift-and-invert. At a given step, we would have a certain basis $W =
[w_1, w_2, \ldots, w_\nu] $ of the current subspace and we apply say
$q$ steps of Algorithm~\ref{alg:shinv} to each column $w_j$.  Each
vector $w_j $ is of the form $w_j = [v_j; \ u_j]$ using  previous
notation.  We perform $q$ such steps of the shift-and-invert method and
denote the $k$-th iterate by $w_j\up{k} = [v_j\up{k}; \ u_j\up{k} ] $.
After a column is processed by these $q$ steps 
\emph{we
discard its top part and extract the $U$-part that will be used
for the projection process.}  In other words, the $j$-th column of the desired
$U$ is simply the bottom part of the vector resulting from $q$ steps
of the shift-and-invert procedure applied to the $j$-th column of $W$.
This is done
\emph{one column at a time} and therefore \emph{we only have to keep one
  vector of length $(m+1)n$.}  Doing this for each  column of $U$ in succession
constitutes one step of what we call ``\emph{reduced subspace  iteration}''.
The resulting technique is summarized in Algorithm~\ref{alg:subsit} which invokes
a \texttt{restart\_vec} function  in line~3 to select a vector $w$ for the shift-and-invert
iteration. This is discussed next.

Since we would like to avoid keeping $\nu$ vectors of length $(m+1) n$, the vector
$w$ in line~3 of the algorithm, which  is used to generate the  $j$-th column of $U$,  is selected as follows.
At the very first outer iteration ($\ell==1$), $w$ is  selected to be a fresh random vector
for each $j$.
In the second (outer)  iteration and thereafter, $w=[v;  \ u]$ should ideally be taken
to  be  an  approximate  eigenvector  of  (\ref{eq:sc} -- \ref{eq:sc1}).
After  the Rayleigh-Ritz  projection  is performed  in lines~6--7, we obtain
$\nu$ approximate  eigenpairs $(\widetilde \lambda_j,  \widetilde u_j)$,
for $j=1,\ldots, \nu$ for the surrogate
problem~\eqref{eq:NLR}.  Each  of the vectors $\widetilde u_j$ 
yields the bottom ($U$-part)  of some approximate
eigenvector $\widetilde w$ associated with the eigenvalue $\widetilde \lambda_j$,
but  the corresponding
$\widetilde v$ vector  (top part of $\widetilde w$)  is not available.
This  is remedied by
relying  on the  relation  \nref{eq:vi},   i.e.,  we  define  the vector $v$
by setting each of its $i$-th components to be
$v_{i} = \widetilde u_j /(\sigma_i - \widetilde \lambda_j)$: 
\[
w = \texttt{restart\_vec} (j) =
\left\{ \begin{array}{rll} 
   \text{if}  \quad    \ell==1 :& \texttt{randn}((m+1)n,1), \\
   \text{else :}           & [v; \widetilde u_j] \ \text{with} \ \ v_i =
   \frac{\widetilde u_j}{\sigma_i - \widetilde \lambda_j },  \ i=1,\ldots, m \ . 
\end{array} \right.  
\]
Note that we preferred to keep the notation
simple by avoiding the extra index $j$ to the vector $v$
(adding the $j$ index would put each subvector
$v_i$ in the form  $v_{ij}$).

\IncMargin{1em}
\begin{algorithm}[h!]
  \SetKwFunction{mgs}{mgs}
  \SetKwInOut{Input}{Input}
  \SetKwInOut{Output}{Output}
  \SetKwInOut{Start}{Start}
  \Input{Subspace dimension $\nu$; $q$; 
          Number of eigenvalues $k$ (with  $k\le \nu$)}
        \Output{$\lambda_1,   \ldots,  \lambda_k$, $U_k$}
        \For{$\ell=1,2, \ldots, $ until convergence:}{ 
          \For{$j=1:\nu$}{
        Select $w = \texttt{restart\_vec} (j)$  \;
           Run $q$ steps of Algorithm~\ref{alg:shinv} starting  with $w$\;
           If $w = [v;\  u]$ is the last iterate, then set $U(:,j)  = u$\;
        }       
          Use  $U$   to  compute   $\widetilde  B_0$,
        $\widetilde   A_0$  and   $\widetilde  B_i,   \ i=1, \ldots,m$  from
        \eqref{eq:RR2}\;
        Compute eigenpairs $\lambda_j, y_j$ of rational problem  \eqref{eq:RR2}, and
        the associated Ritz vectors $u_j = U y_j$ for $j=1,\ldots, \nu$\;
        }
	\Return{$\lambda_1, \ldots, \lambda_k$ and eigenvector matrix $U_k$}
	\caption{Reduced  Subspace Iteration}\label{alg:subsit}
\end{algorithm}\DecMargin{1em}

\subsection{A gradual precision procedure}
An attractive feature of the procedure described in the previous
section is that sometimes it is possible to select the quadrature points
in such a way that several approximations are available from the same
set of (fine) quadrature points. This can be exploited in
Algorithm~\ref{alg:subsit} by using gradually more accurate rational
approximations as the outer loop progresses.  For example, if we use a
total of $16$ points in a trapezoidal rule, as illustrated in
Figure~\ref{fig:mlev}, we can use the set of points $0, 4, 8, 12$ at the
very first iteration, then the set of all even points at the next, and
then all points at the $3$rd outer iteration.  In a multilevel
generalization of this scheme with $L$ outer iterations in Algorithm~\ref{alg:subsit} 
-- we will use $n_0 \times 2^{\ell - 1}$ points at level $\ell$.
Each set of points will lead to a pair of
matrices $\ca{A}\up{\ell},\ca{M}\up{\ell}$ in the linearization
\ref{eq:sc} that increase in size as $\ell$ increases.
The idea that is  exploited here is that the two consecutive rational approximations of $T(z)$
are close to each other, so initial vectors for the fine approximations (more quadrature points)
can be obtained from 
coarser ones (fewer quadrature points) to build approximate eigenvectors in a progressive way.

To achieve this, the only change that is needed in
Algorithm~\ref{alg:subsit} is to perform  the shifted inverse iteration in
line~4 with the pair $\ca{A}\up{\ell},\ca{M}\up{\ell}$ instead
of $\ca{A},\ca{M}$.  If more than $L$ iterations are needed, one can
continue iterating with the most accurate pair, i.e., the last pair
$\ca{A}\up{L},\ca{M}\up{L}$.  The primary motivation here is to reduce
the number of outer iterations required when iterating with the most
accurate pair.  In the case when the sets of quadrature points are
nested, e.g. for the trapezoidal rule, the storage and
computational costs are minimized as is discussed next.

\begin{figure}[H]
  \centerline{
    \includegraphics[width=0.3\textwidth]{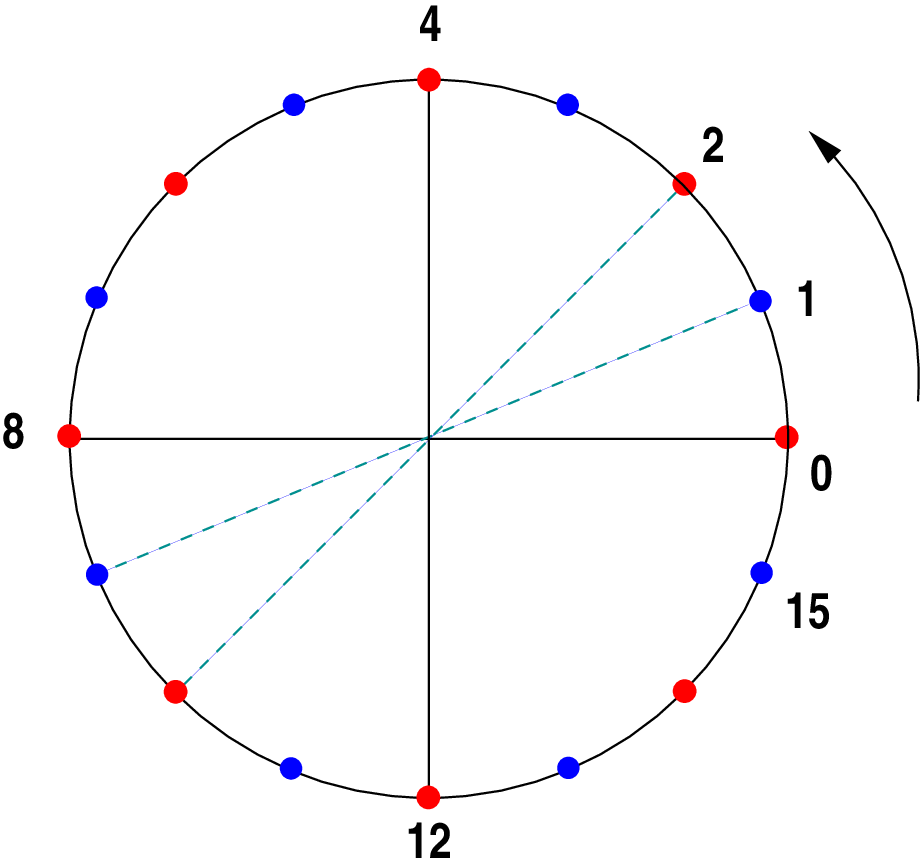}
  }
  \caption{Using 16  quadrature points for the trapezoidal rule.
    The set of nodes in the 4-point quadrature is a subset of 
    the nodes of the 8-point quadrature which is itself a subset of the
    the nodes of the 16-point quadrature rule. }\label{fig:mlev}
\end{figure}

To implement this scheme efficiently, it is best to
write the approximation of $f_j(z)$  as
\eq{eq:fjM}
f_j(z)  \approx  \sum_{i=1}^m \frac{\omega_{i}  f_j (\sigma_i) }{z - \sigma_i}.
\en
Then  expression \eqref{eq:Tofz1}  becomes:
\begin{align}  
\widehat T(z) &= -B_0+ z A_0 + 
\sum_{j=1}^p \sum_{i=1}^m
                  \frac{\omega_{i}  f_j(\sigma_i) }{z - \sigma_i} A_j =
                  -B_0+ z A_0 + \sum_{i=1}^m \ \omega_i
\frac{ \sum_{j=0}^p f_j (\sigma_i)  A_j }{z - \sigma_i}  \nonumber \\
        &\equiv  \ -B_0+ z A_0 + 
\sum_{i=1}^m \ \omega_i 
\frac{ \widehat B_i }{z - \sigma_i},  \label{eq:TofzMM}
\end{align}
where we now set:
\eq{eq:TBi} \widehat B_i = \sum_{j=0}^p f_j(\sigma_i) A_j, \qquad
i=1, \ldots, m.  \en In other words $B_i = \omega_i \widehat
B_i$.
Note here that although
the weights $\omega_i$ depend on the level $\ell$, the matrices
$\widehat B_i$ are fixed.  In fact when $p$ is small - as is
commonly the case - it is best not to store the $\widehat{B}_i$'s
(or the $B_i$'s). Indeed, the matrix $\widehat{B}_i$ (as well as 
$B_i$) is  only invoked for matrix-vector products (``matvec's'')
which can be carried out with $p$ matvecs using  the original matrices
$A_i$ plus a linear combination of $p$ vectors. If $p$ is large, then
it may be advantageous to compute and store the $\widehat B_i$'s.

\section{Theoretical considerations} \label{sec:theory} 
Let us consider problem (\ref{eq:sc} -- \ref{eq:sc1}) under the
simplified assumption that $A_0 = I$. This is equivalent to having an
$A_0$ that is invertible, since in this case we can multiply
equation~\nref{eq:Tofz} by $A_0 \inv $ to reach the desired form in
which $A_0 = I$. Hence, without loss of generality, we assume $\cmm =
I$.

\subsection{Characterization of eigenvalues of $\cma$} 
Now, we would like to examine all the eigenvalues of matrix
$\cma$. For this, we consider the characteristic polynomial of matrix
$\cma$. We write $\cma - z I $ as follows: 
\eq{eq:cmamz}
\cma - z I = \begin{bmatrix} D - z I & F \\ B^T & B_0 - z I  \end{bmatrix},
\en
where, referring to \nref{eq:sc1}, we see that 
$D$ is $(mn) \times (mn)$, and
$F$ and $B$ are both $(mn) \times n$.
Then, when $D-z I$ is invertible,
i.e., when $z$ is different from all the $\sigma_i$'s, 
the block LU factorization of $\cma - z I$ is given as:
\eq{eq:cmamz1}
\cma - zI = \begin{bmatrix}  
	I & 0 \\
	B^T (D-zI)^{-1} & I
        \end{bmatrix} 
        \begin{bmatrix} 
        D-z I  & F \\
0 & S(z) 
\end{bmatrix},
\en
where $S(z) $ is the (spectral) Schur complement:
\eq{eq:Schur}
S(z) \equiv B_0 - zI - B^T (D-z I) \inv F =
B_0 - zI + \sum_{i=1}^m \frac{B_i}{\sigma_i - z}.
\en 
Hence, when $z$ is not a pole, then
$\det (\cma - z I) = \det(S(z)) \det(D- z I).$
Moreover, by comparing equations~\nref{eq:Tofz1} and \nref{eq:Schur}, we
observe that $S(z) = - \widetilde T(z)$.

Assume now that $z$ is a pole, e.g., without loss of generality let $z = \sigma_1$.
In this case, we can use a continuity argument. Indeed,
$\det (\cma - z I) $ is a continuous function and therefore we can define $\det(\cma - z I)$ as the limit:
\begin{align*}
  \det(\cma - \sigma_1 I) &= \lim_{z \tendsto \sigma_1}
  \det\left[B_0 - zI + \sum_{i=1}^m \frac{B_i}{\sigma_i - z}\right] \prod_{i=1}^m (\sigma_i - z)^n\\
                   & = \lim_{z \tendsto \sigma_1} \det\left[ (\sigma_1-z) (B_0 - z I)
                     + \sum_{i=1}^m \frac{\sigma_1 -z}{\sigma_i - z} B_i\right] \prod_{i=2}^m (\sigma_i - z)^n\\
                   & =  \det ( B_1 ) \prod_{i=2}^m (\sigma_i - \sigma_1)^n . 
\end{align*}
Note that a second approach to prove the above relation  is to observe that when
$z = \sigma_1$, then the top left $n\times n$ block of $\cma - z I$ is a zero block and this
can be exploited to expand the determinant.
This result can be generalized to any other $\sigma_i$, and therefore we can state the following lemma.
  \begin{lemma}\label{lem:det} The following equality holds~:
    \eq{eq:detA}
    \det (\cma - z I) =
    \left\{ \begin{array}{lcl}
              \det (S(z)) \prod_{j = 1}^m (\sigma_j - z)^n & if & z \ne \sigma_i \ \ i = 1,\ldots m, \\
              \det ( B_i ) \prod_{j\ne  i}^m (\sigma_j - \sigma_i)^n
                           & if & z = \sigma_i.
            \end{array}
          \right.
          \en
    \end{lemma}

We denote by $e_i$ the $i$-th canonical basis vector of the vector space $\CC^{m+1}$ 
and by $\otimes$ the \emph{Kronecker} product (operator \texttt{kron} in {\sc Matlab}).
The following corollary is an immediate consequence of Lemma \ref{lem:det}.

    \begin{corollary} If all the matrices $B_i, i=0,\ldots,m$ are nonsingular, then
      the eigenvalues of \nref{eq:NLR} are the same as the eigenvalues of matrix $\cma$.
      If a matrix $B_i$ is singular and $u$ is an associated null vector, then
       $\sigma_i$ is an eigenvalue of $\cma$ and $e_i \otimes u$ is an associated eigenvector.
     \end{corollary}
     
Hence, we can  ignore any eigenvalue that is equal to one of the $\sigma_i$'s when it occurs. 
The next results will show that as long as the rational approximations of the functions $f_j: \Omega \rightarrow \CC$ are accurate enough,
the eigenvalues of $\widetilde T(z)$ will be good approximations to all eigenvalues of $T(z)$ located inside the region $\Omega$. 
    

\subsection{Accuracy of computed eigenvalues} 
Let $\Omega_1 $ be a region strictly included in the (larger) disk $\Omega $
such that $\| f_j(z) -r_j(z) \|_{\Omega_1} <\varepsilon $, where the $\Omega_1$-norm is, e.g.,
the infinity norm in $\Omega_1$ and $r_j(z)$ is the rational approximation of function $f_j(z)$. In other words,
each function $f_j(z)$ in \nref{eq:Tofzm1} is approximated by a rational function $r_j(z)$ and this approximation
is assumed to be accurate within an error of $\varepsilon$ in the region $\Omega_1$. Our goal now is to show
that each of the eigenvalues inside $\Omega_1$ is a `good' approximation to an eigenvalue of the original problem
\eqref{eq:nlin0}. This can be done by exploiting the corresponding approximate eigenvectors 
and by considering the residual associated with the approximate eigenpair.
The following simple proposition shows a result along these lines.
\begin{proposition}\label{prop:1} 
  Let us assume that $\| f_j(z) - r_j(z)\|_{\Omega_1} \le \varepsilon$
  for $j=1,\cdots,p$  and let $(\widetilde  \lambda, \widetilde u)$ be  an exact eigenpair
  of
  the  surrogate  problem  \nref{eq:NLR} with  $\widetilde  \lambda$
  located  inside $\Omega_1$  and $  \|  \widetilde u  \| =  1$ for  a
  certain      vector      norm      $\|\cdot     \|$.      \      Let
  $ \mu = \sum\limits_{j=1}^p \|A_j\|$.  Then,
\[ 
\| T(\widetilde \lambda) \widetilde u \| \le \mu \varepsilon.
\]
\end{proposition}

\begin{proof}
The approximate problem \nref{eq:NLR} is obtained by replacing $T(z)$ in \nref{eq:Tofz} by:
\eq{eq:Rofzm}
\widetilde T(z) =   -B_0 + z A_0 + r_1(z) A_1 +\ldots + r_p(z) A_p.
\en 
Since $(\widetilde \lambda, \widetilde u)$ is an eigenpair of problem \nref{eq:NLR}, $\widetilde T(\tla) \widetilde u = 0$, which implies:
$$   [-B_0 + \tla A_0 + r_1(\tla) A_1 +\ldots + r_p(\tla) A_p] \tilde u = 0.$$
Setting $f_j(\tla) - r_j(\tla) = \eta_j(\tla)$ and 
substituting this into above equation gives:
\begin{align*}
  [ -B_0 + \tla A_0 & +  \sum_{j=1}^p (f_j(\tla) - \eta_j(\tla))  A_j ] \tlu = 0, \\ 
[ -B_0 + \tla A_0 & +  \sum_{j=1}^p f_j(\tla) A_j ] \tlu  = \left[\sum_{j=1}^p \eta_j(\tla))  A_j \right] \tlu,
\end{align*}
and thus
$
T(\tla) \tlu =  \left[\sum_{j=1}^p \eta_j(\tla)  A_j \right] \tlu.
$ 
Taking norms on both sides and recalling that
$\| f_j(z) - r_j(z)\|_{\Omega_1} \le \varepsilon$  yields
$ 
  \| T(\tla) \tlu \|   = \|\sum_{j=1}^p \eta_j(\tla)  A_j \tlu \|  \ \le \mu
  \varepsilon ,
  $ which is the desired result.
\end{proof}

Proposition    \ref{prop:1}   implies    that    any   eigenpair    of
problem~\nref{eq:NLR}  is an  approximate  eigenpair  of the  original
problem provided that each function $f_j(z)$ is well approximated by a
rational  function  $r_j(z)$ in  $\Omega_1$  and  that the  eigenvalue
$\tla $ is  inside $\Omega_1$. By a backward argument  the opposite is
also true.
\begin{proposition}
  Let us assume that $\| f_j(z) - r_j(z)\|_{\Omega_1} \le \eps$
  for $j=1,\cdots,p$ and let $(\lambda, u)$ be an exact eigenpair for $T(z)$ with
$\lambda$ located inside $\Omega_1$ and  $\| u \|  = 1$.
Then, $ (\lambda, u)$ is an approximate eigenpair of the problem \nref{eq:NLR}, i.e., 
\[
\| \widetilde T(\lambda) u \| \le \mu \varepsilon, 
\]
where $\mu $ is defined as in Proposition~\ref{prop:1}.
\end{proposition}

\begin{proof} 
The proof is essentially identical to that  of Proposition~\ref{prop:1}.
\end{proof}

These two results show that for $\varepsilon$ small enough, we should be able to find approximations 
to all eigenvalues of the exact problem located in $\Omega_1$ (and only these) by solving \nref{eq:NLR},
except for cases of highly ill-conditioned eigenvalues. 

\subsection{Conditioning of a simple eigenvalue}

For the reason stated above, it is of particular importance to examine
the   condition   number   of    an   eigenvalue   of   the   extended
problem~(\ref{eq:sc}  --  \ref{eq:sc1}).  There is no loss of generality in   assuming
that $A_0  = I$.   Let us  consider a  simple eigenvalue  $\lambda$ of  the
matrix $\cma$. Its right eigenvector is a vector $w$ of the form shown
in \nref{eq:Block}  with $v_i = u  / (\sigma_i - \lambda)$  defined in
\nref{eq:vi}. As seen earlier --  equation \nref{eq:NLR} -- the vector
$u$   is    an   eigenvector   of   $S(\lambda)$,    i.e.,   we   have
$ S(\lambda) u = 0 .$

The left eigenvector is a vector $s$ that satisfies
$\cma ^H s = \overline \lambda s $. Similarly to $w$, it  consists of block components $h_1, \ldots, h_m, $ and $y$.
The equation
$(\cma ^H - \overline \lambda I) s = 0  $ yields the relations:
\[
  (\overline \sigma_i - \overline \lambda ) h_i + B_i^H y = 0
\quad \Rightarrow  \quad h_i = \frac{1}{\overline \lambda - \overline \sigma_i} B_i^H y,   \quad i=1,\ldots, m
\]
and
\[
  - \sum_{i=1}^m h_i + (B_0^H -  \overline \lambda I) y = 0  \quad \Rightarrow \quad
  S(\lambda)^H y = 0 .
  \] 
  Thus,  the right and the left eigenvectors of $\cma$ are defined in terms of the right and the left eigenvectors
  $u$ and $y$ of $S(\lambda)$.

  As is well-known, the condition number of a simple eigenvalue $\lambda$ is the inverse of the cosine of
  the acute angle between the left and the right eigenvectors.
  Before  considering the inner product $(w, s)$, we point out  that the derivative of $S(z)$ is:
  \eq{eq:Sprime}
  S'(z) = - I + \sum_{i=1}^m \frac{B_i}{(z-\sigma_i)^2} .
  \en
  Consider  now the inner product $s^H w$:
  \[
    s^H w = y^H u + \sum_{i=1}^m h_i^H v_i = y^H u - \sum_{i=1}^m  \frac{y^H B_i u}{(\lambda-\sigma_i)^2}
    = - y^H S'(\lambda) u .
    \]
    Finally, we need to calculate the norms of $s$ and $w$. For $w$ we have:
    \[
      \| w \|_2^2 = \| u \| ^2 + \sum_{i=1}^m
      \frac{\|u \|_2^2 }{|\lambda-\sigma_i|^2} = \| u \|^2 \left[ 1 +
              \frac{1}{|\lambda-\sigma_i|^2} \right] , \]
while for $s$:
    \[
      \| s \|_2^2 = \| y \| ^2 + \sum_{i=1}^m
      \frac{\|  B_i y \|_2^2 }{|\lambda-\sigma_i|^2} = \| y\|^2 \left[ 1 +
        \sum_{i=1}^m  \frac{\| B_i y\|_2^2 }{\|y \|_2^2 |\lambda-\sigma_i|^2} \right] . \]
    Assuming that the vectors $u$ and $y$ are of norm one
    leads to  the following proposition which establishes an expression for the desired condition number. 

          \begin{proposition}\label{prop:cond}
            Let $\lambda$ be a simple eigenvalue of \nref{eq:sc1}, and $u, y$ the corresponding unit norm right and left
            eigenvectors (respectively) of  $S(\lambda)$. Then the condition number of $\lambda$ as
            an eigenvalue of (\ref{eq:sc} -- \ref{eq:sc1}) is given by
            \eq{eq:condA}
            \kappa(\lambda)  = 
            \frac{ \alpha_u \alpha_y } { | (S'(\lambda) u, y)| },
            \en
            where 
\eq{eq:condA1}
            \alpha_u = \sqrt{1 + \sum_{i=1}^m \frac{1}{|\lambda - \sigma_i|^2}} \qquad \text{ and } \quad \quad
            \alpha_y = \sqrt{1 + \sum_{i=1}^m  
              \frac{\| B_i y\|_2 }{|\lambda - \sigma_i|^2}} .      \en
          \end{proposition}

          The coefficients $\alpha_u,  \alpha_y$ will remain bounded as long as $\lambda$ is far away from
          any of the poles.
          Note in particular that the terms $\| B_i y \|_2 $ can be bounded by the constant
          $\beta = \max_i \|B_i\|_2 $.
          On the other hand nothing will prevent  the denominator  in \nref{eq:condA}  from being close
          to zero. The condition number can be easily gauged to determine if this is the case. Note that
          $y^H S'(\lambda) u $ is inexpensive to compute once the eigenvectors $u$ and $y$ are available.

\subsection{The halo of extraneous eigenvalues}
In all our experiments we observed that the eigenvalues  of the problem
(\ref{eq:sc} -- \ref{eq:sc1}) that are not eigenvalues of the
          original nonlinear problem \eqref{eq:Tofz} tend to congregate into a `halo' around the
          contour $\Gamma$ used for the Cauchy integration. 
          
            It is possible to explain this
          phenomenon.
          First note that the basis of the method under consideration is to approximate the original nonlinear matrix
          function $T(z)$ in \nref{eq:Tofz}  by the rational function:
\eq{eq:Rofz}
\widetilde T(z) =  - B_0 +  z A_0 + r_1(z) A_1 + \ldots + r_p(z) A_p , 
\en
where each $r_j(z)$ is a rational approximation of $f_j(z)$.

 Consider the situation when
$z$ is outside the domain used to obtain the Cauchy integral,  far from the contour.
Assuming that the number of quadrature points $m$ is large enough, then each $r_j(z)$ will be close to zero.
Hence, any eigenvalue of
$\widetilde T(z) $ that is outside the contour and not too close to it should be just
an eigenvalue of the generalized problem $(B_0 - \lambda A_0) u = 0$.
In other words, it should be close to an eigenvalue of the linear part of $T(z)$.

Let us now consider the opposite case of an eigenvalue of $\widetilde T(z)$ that is inside
the contour but also not too close to it.  Our earlier results show that in this case
we should only find eigenvalues of the original problem \eqref{eq:Tofz} and no other eigenvalues.

For an eigenvalue to be \emph{extraneous}, i.e., in the spectrum of
(\ref{eq:sc} -- \ref{eq:sc1}) but not of \nref{eq:Tofz}, it must therefore \emph{either
be (close to) an eigenvalue of the linear part of $T(z)$, or
located (close to) the contour.} 

Although  this argument  is based  on a  simple model,  it provides  a
picture that is remarkably close to what is observed in practice. Next we  present  an  illustration   using  a  small  quadratic  eigenvalue
problem. Quadratic eigenvalue problems can be handled more efficiently
by  standard  linearization  than  by the  method  presented  in  this
paper.  However, they  can be  useful  for the  purpose of  validation
because their eigenvalues are  readily available. Consider the problem
\eq{eq:QuadEx} (-B_0  + \lambda A_0 +  \lambda^2 A_2) u =  0, \en where
the matrices $B_0, A_0, A_2$ are generated by the following three {\sc
  Matlab} lines of code with $n=4$:
\begin{verbatim}
  B0 = -2*eye(n)+diag(ones(n-1,1),1)+diag(ones(n-1,1),-1) ;
  A0 = eye(n); 
  A2 = 0.5*(n*eye(n)-eye(n,1)*ones(1,n)-ones(n,1)*eye(1,n));
\end{verbatim} 
The eigenvalues  are all located  inside a rectangle  with bottom-left
and top-right  corners $(-1, -1.5i),  (0, 1.5i)$  which we use  as the
integration contour.

\begin{figure}[h]
  \includegraphics[width=0.45\textwidth]{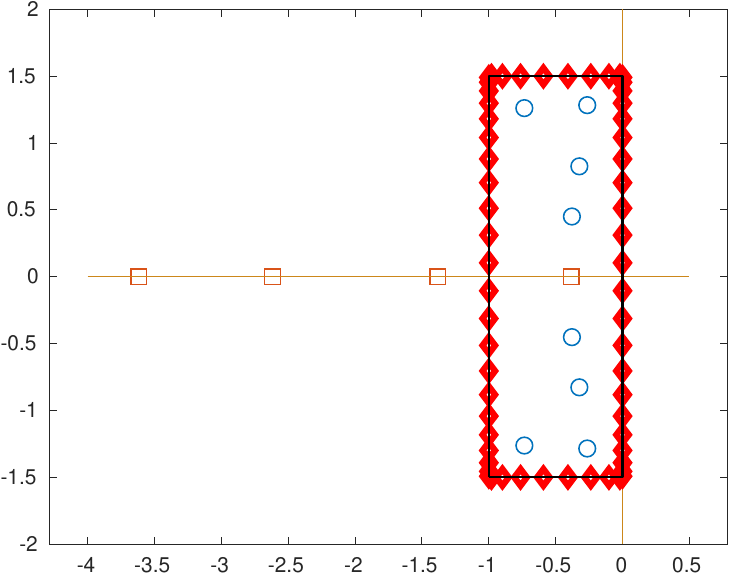}
  \includegraphics[width=0.45\textwidth]{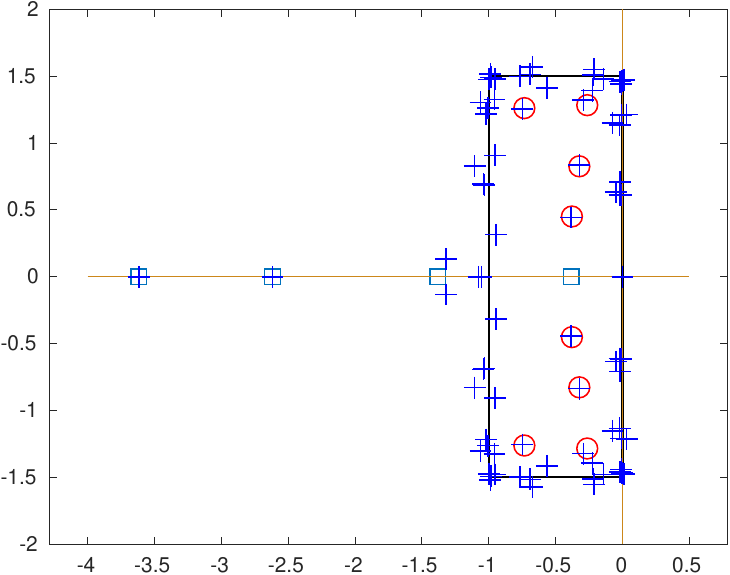} 
  \caption{
    Left: The $8$ eigenvalues of the original problem \eqref{eq:QuadEx} (circle); 
    the $4$ eigenvalues of the linear
    part (square); contour and quadrature points along it.
    Right: Eigenvalues computed with $m=20$ quadrature points (plus)
    along with contour, original eigenvalues (circle), and eigenvalues of
    linear part (square). \label{fig:quad20}
   }
  \end{figure}

  \begin{figure}[h]
  \includegraphics[width=0.45\textwidth]{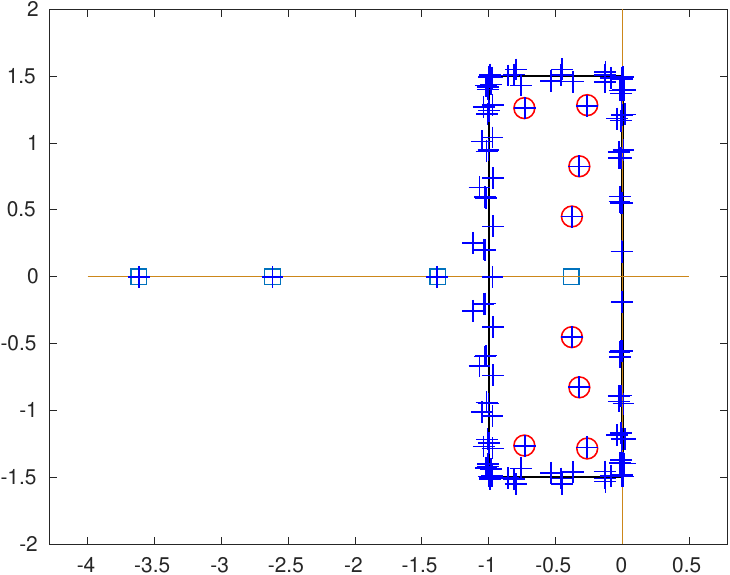}
  \includegraphics[width=0.45\textwidth]{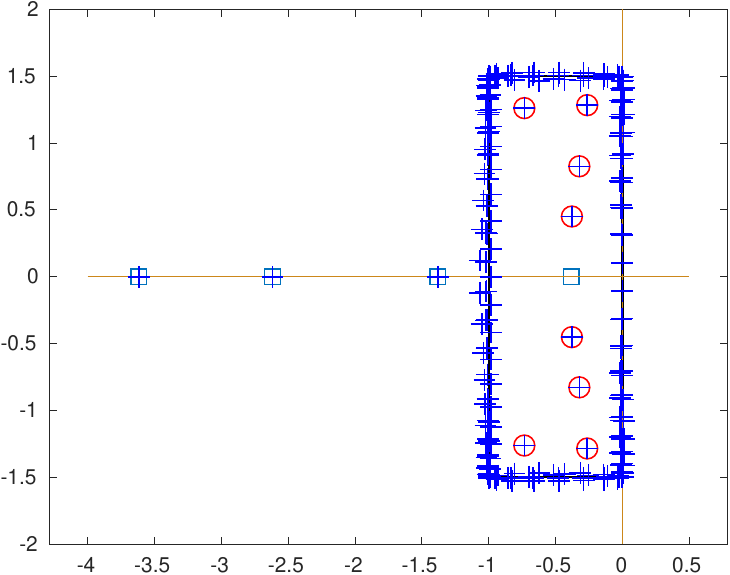}
  \caption{Same information as in the right part of  Figure~\ref{fig:quad20}
    using a total of $m=32$ quadrature points (left) and
     $m=60$ quadrature points (right). \label{fig:quad60}
  }
\end{figure}

Gauss-Legendre quadrature  formulas are  invoked on  each side  of the
rectangle with a  number of points selected to be  proportional to the
side length.  The  left side of Figure~\ref{fig:quad20}  shows the $8$
eigenvalues of the original problem \eqref{eq:Tofz} as well as the $4$
eigenvalues of the  pencil $(B_0, A_0)$. Three of  these eigenvalues are
located  outside  the  contour  and  one inside.  The  right  part  of
Figure~\ref{fig:quad20} and the two  plots in Figure~\ref{fig:quad60} show
the  eigenvalues of  problem  (\ref{eq:sc} --  \ref{eq:sc1}) when  the
number of contour points is $m=20, 32, 60$, respectively.  When $m=20$
the  approximations are  still rough.  However, the  pattern mentioned
above begins  to unravel: the eigenvalues  of \nref{eq:QuadEx} located
inside  the  contour   are  more  or  less   approximated,  and  those
eigenvalues of $(B_0, A_0)$ that are outside are starting to be neared
by pluses.  Observe that the eigenvalue  of $(B_0, A_0)$ that  is near
$-1.5$ is approximated by two  eigenvalues of $\cma$. In contrast, the
one near $-0.5$ (inside the  contour) is not approximated as predicted
by the  theory.  As $m$ increases  this picture is confirmed:  (1) all
the  eigenvalues  of  \nref{eq:QuadEx}  inside the  contour  are  well
approximated,  (2) all  the  eigenvalues of  $(B_0,  A_0)$ outside  the
contour are  well approximated by  eigenvalues of $\cma$, and  (3) the
eigenvalue  of   $(B_0,  A_0)$  inside  the   contour  is  essentially
`ignored'.  In  addition, the halo  of eigenvalues around  the contour
becomes quite close to the contour itself.  This small example provides a
good  illustration of  the general  behavior  that we  observe in  our
experiments.

%
%


\section{Numerical Experiments} \label{sec:exper} 
All  the   numerical  experiments  presented  in   this  section  were
performed with {\sc Matlab} R2018a. We will illustrate the behavior of
Algorithm  \ref{alg:subsit}  on
several  nonlinear  eigenvalue   problems  discussed  in  \cite{Bey12,
Ruh73, Kre09}. All the examples considered come in the form
given in \nref{eq:Tofz}. For most  of the examples, the contour $\Gamma$
is either   circular or  rectangular and we seek  the eigenvalues
closest   to  the   center of $\Gamma$.
For   Algorithm \ref{alg:subsit}, the shift $\sigma$ is selected to be the center of the
region enclosed by the contour $\Gamma$.

In the case of a circular contour,  the $m$ quadrature nodes and weights
used  to  perform  the  numerical integration  to approximate  the
functions $f_j$   inside the  contour $\Gamma$ were  generated using 
the Gauss-Legendre quadrature rule. To choose  a suitable $m$, we take
two circles  $\Omega_1$  and $\Omega$  with the  same center
and $\Omega_1 \subset \Omega$. Then,  $m$ is increased until the accuracy
of  the resulting rational  approximation is  high enough  inside
$\Omega_1$. Note that we need to  avoid a region near the outer circle
not only because of the poles,  but also because the approximations of
the $f_j$'s  will tend to be  poor in this region.   To illustrate the
effectiveness of  the proposed approaches, we  compare the obtained eigenvalues with  the  ones   obtained  by  Beyn's
method~\cite{Bey12} or/and via a corresponding linearization. We also compare our algorithm with some well-established nonlinear eigensolvers utilizing rational approximation, i.e., the set-valued AAA algorithm~\cite{NakST18} and the NLEIGS~\cite{GueVBMM14}.

%
%

\subsection*{Example 1}
Consider the following example discussed in \cite[Sec. 2.4.2]{MicN07} and \cite[Example 13]{Kre09},
\noindent
\eq{eq:Kresex}
	T(z) = -B_0+zI+e^{-z\tau}A_1,
\en
with $B_0 = \begin{pmatrix} 
	-5 & 1\\
2 & -6
\end{pmatrix}$, $A_1=-\begin{pmatrix} 
-2 & 1 \\
4 & -1 
\end{pmatrix}$ and $\tau = 1$.
The   nonlinear    eigenvalue   problem   \nref{eq:Kresex}    is   the
characteristic       equation       of      a       delay       system
$x'(t) = - B_0x(t)+A_1x(t-\tau)$.  For the purpose of  a comparison with
the   results  from~\cite[Example   5.5]{Bey12},   we  calculate   all
eigenvalues enclosed  by a  circle centered  at $c  = -1$  with radius
$r = 6$.   Referring to Proposition~\ref{prop:1}, we  first check which
values of $m$  will provide a good rational  approximation $r_m(z)$ of
$f(z)=e^{-z}$.   The right  part of  Figure~\ref{fig:Kres1} shows  the
errors $e_m=\|f(z)-r_m(z)\|_{\infty}$
(evaluated on a finely discretized  version of  $\Omega_1$)
versus  the number of quadrature
nodes $m$.  Notice that the  accuracy of the rational approximation of
$f(z)$ inside the considered contour is  good enough for $m = 50$.  We
can  therefore solve  the eigenvalue  problem \nref{eq:sc},  associated
with the approximate problem  \nref{eq:NLR}, with $m=50$ Gauss-Legendre
quadrature nodes. The left  part of  Figure~\ref{fig:Kres1} compares
the  eigenvalues  computed by applying the shift-and-invert Arnoldi method  to the large system
(\ref{eq:sc} -- \ref{eq:sc1}) with $m = 50$ quadrature nodes  and those
computed  by Beyn's  method using  the  same contour.
Note that since
the number of eigenvalues in the considered contour is larger than the
size of problem~\eqref{eq:Kresex}, we use Beyn's second algorithm
with three moments to compute the five eigenvalues with the backward error 
smaller than $\delta = 10^{-10}$. With these parameters, the number of quadrature nodes
for Beyn's method necessary to get the five eigenvalues inside the circle is $80$. 
\begin{figure}
	\centering
  \includegraphics[width=0.45\textwidth]{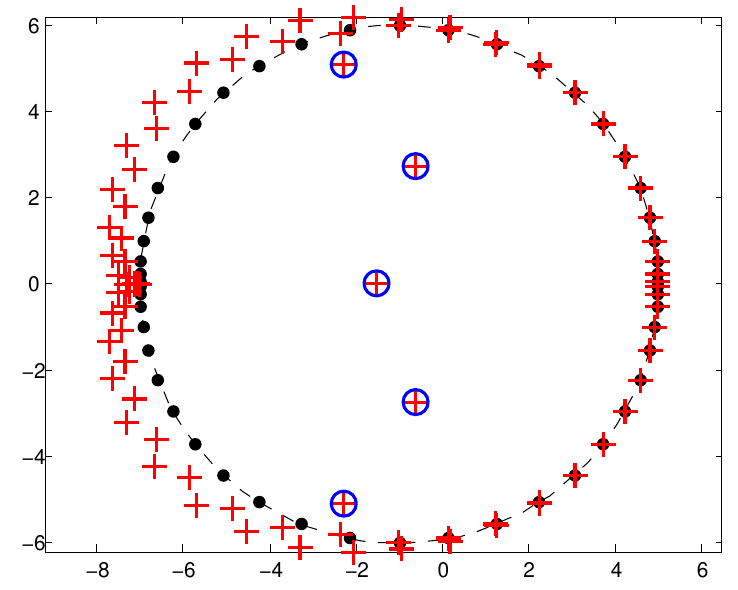}
  \includegraphics[width=0.45\textwidth]{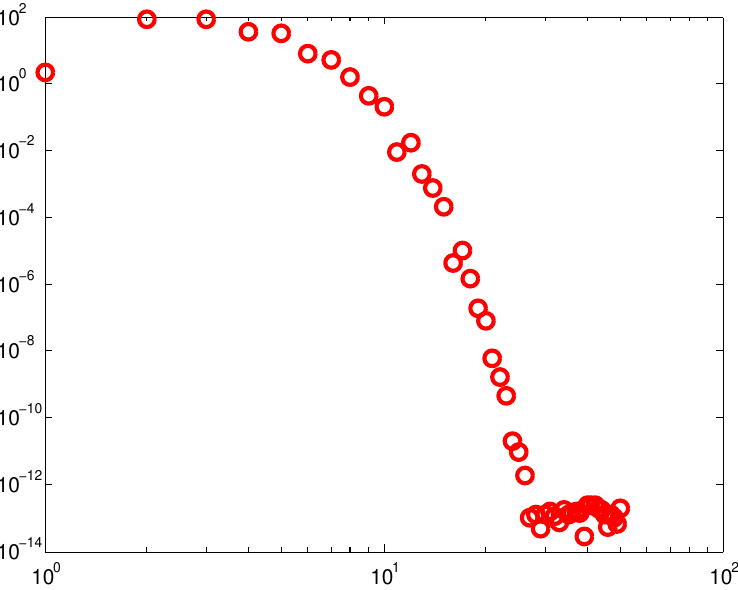}
  
  \caption{Left: Eigenvalues of \nref{eq:Kresex} inside a
    circle of radius $r = 6$ centered at $c = -1$ obtained by solving the eigenvalue problem \nref{eq:sc} (plus)
    using $m=50$ quadrature nodes (point) and by Beyn's method (circle) using $150$ quadrature nodes.
    Right: The errors $e_m$ of the rational approximation of $e^{-z}$ versus the number of quadrature nodes $m$.}
  \label{fig:Kres1}
\end{figure}

In  order to  check the  conclusions of  Proposition \ref{prop:1},  we
compute again all the  eigenvalues  of the  generalized eigenvalue  problem
\nref{eq:sc}  associated with  the  approximate problem  \nref{eq:NLR}
inside a  circle $\Omega$ centered  at $c = -1$  with radius $r  = 6$,
using $m=100$  quadrature nodes.   These eigenvalues  are shown  on the
left side  of Figure \ref{fig:Kresall}. Let $\Omega_1$ be a disk with
the same center as $\Omega$ and with radius $r_1=r/2$. Let $\lambda_1$
be the closest eigenvalue to $c$  located in $\Omega_1$ and  $u_1$ be
the corresponding eigenvector.  Recall  that $u_1$ is taken from the last
$n$  entries  of  the  eigenvector  corresponding  to  the  eigenvalue
$\lambda_1$ of \nref{eq:sc}.   Let $\mu$ be the  constant from Proposition
\ref{prop:1} and  $r_m(z)$ the rational approximation  of the function
$f(z)=e^{-z}$.   The  right side  of  Figure  \ref{fig:Kresall}
 compares      the     residuals
$        \|T(\lambda_1)        u_1\|_{\infty}$  with the         errors
$e_m= \mu \|f(z)-r_m(z)\|_{\infty}$
when  the  number of quadrature nodes $m$ varies.

\begin{figure}
	\centering
  \includegraphics[width=0.45\textwidth]{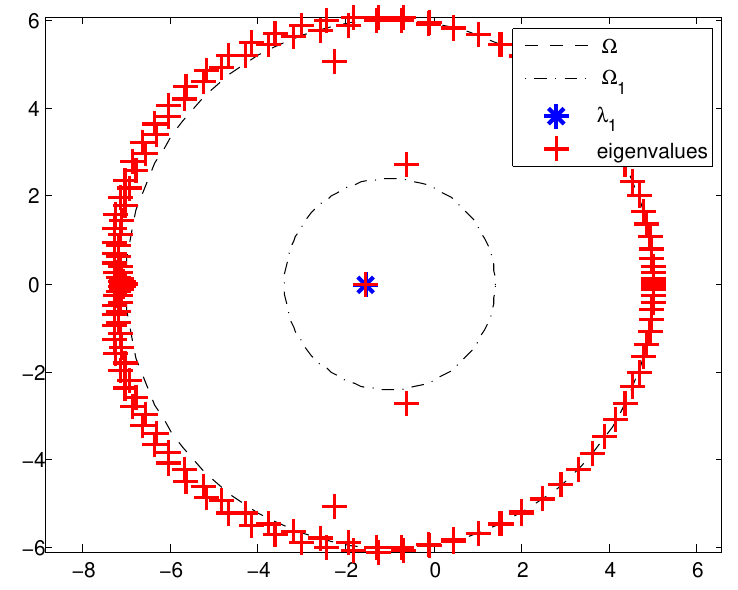}
  \includegraphics[width=0.45\textwidth]{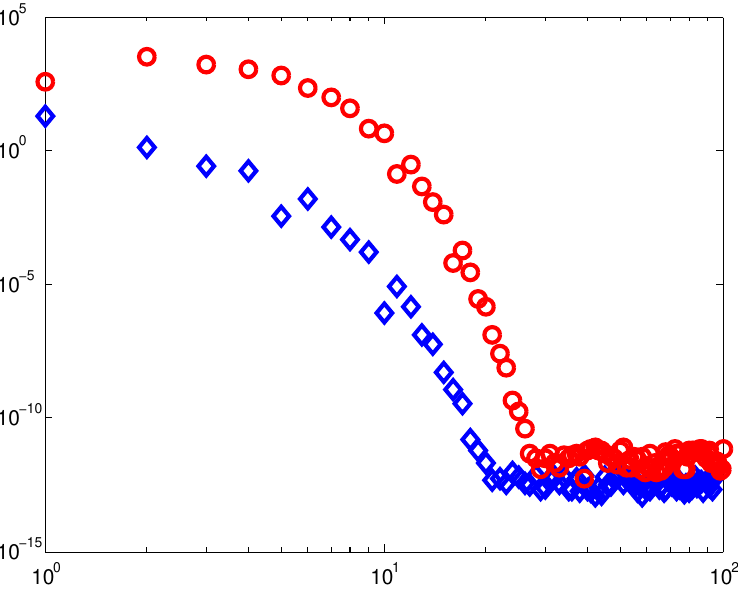}
  \caption{Left: All eigenvalues of \nref{eq:Kresex} (plus) computed via \nref{eq:sc}.   
    Right: Residuals $ \|T(\lambda_1) u_1\|_{\infty}$ (diamond) and errors
    $e_m$ (circle) versus the number of quadrature nodes $m$.}
\label{fig:Kresall}
\end{figure}
%
%

\subsection*{Example 2}
In this experiment, we consider the same nonlinear eigenvalue problem
as in Example 1 with a different search contour.  The location of the
eigenvalues in Example 1 suggests that it should be more effective to
consider a rectangular contour instead of a disk.  An advantage of
rectangular regions is that they are easier to subdivide into smaller
rectangular regions than disks.  For example, we can split a rectangle
in the complex plane into different sub-rectangles and then apply
Algorithm \ref{alg:subsit} in each sub-rectangle. A side benefit of
this approach is the added parallelism since each sub-rectangle can be
processed independently.  Finally, this divide-and-conquer approach
also allows to take advantage of the trade-off between using smaller
regions which require fewer poles versus larger regions which will
yield more eigenvalues at once at the cost of using more
poles.
If $c_1, c_2, c_3, c_4$ are the four corners of the rectangle,
listed counter clock-wise with $c_1$ being the top-left corner, 
the integration starts  at $c_1$, and is performed
counterclockwise using Gauss-Legendre quadrature on each side.
To solve the nonlinear eigenvalue problem \nref{eq:Kresex},
we consider the rectangle defined by the two opposite corners at 
$c_2=-3-6\imath$ and $c_4=1+6\imath$, and we solve the eigenvalue problem \nref{eq:sc}
directly, using $m=40$ quadrature points.
The left part of Figure \ref{fig:KresRectangle}
shows that all eigenvalues are well approximated  
inside the rectangle. To illustrate the behavior of
the rational approximation method for the nonlinear eigenvalue problem
near the nodes, we consider a smaller rectangle defined by
$c_2=-2.5-6\imath$ and $c_4=-0.2+6\imath$.   As we can clearly see on the right side of
Figure \ref{fig:KresRectangle}, 
eigenvalues near the quadrature nodes are not well approximated.
This is a consequence of the poor rational approximation of the
function $f(z)=e^{-z}$ near the quadrature nodes. 
Good eigenvalue approximations can be computed 
by increasing the number of quadrature nodes on each side.
This is illustrated in Figure \ref{fig:KresRectangleIncreasem} in which $m=50$ quadrature nodes are considered.


Even though for some problems rectangular (non-circular)
contours may seem to be better suited, they do not yield 
an exponential increase in accuracy as the number of quadrature points increases as is the case for
the trapezoidal rule on circular contours~\cite[Theorem 2.1]{TreW14}.
They are also  more tedious to implement.

\begin{figure}
	\centering
		\includegraphics[width=0.45\textwidth]{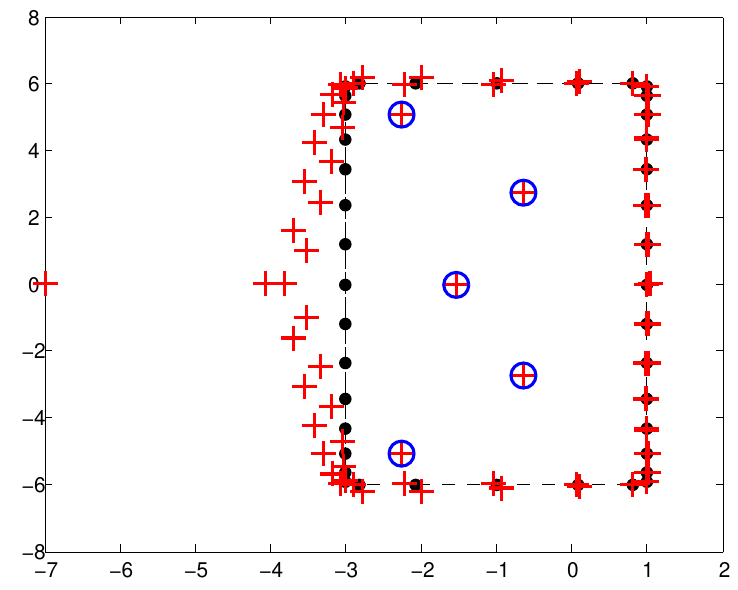}
		\includegraphics[width=0.45\textwidth]{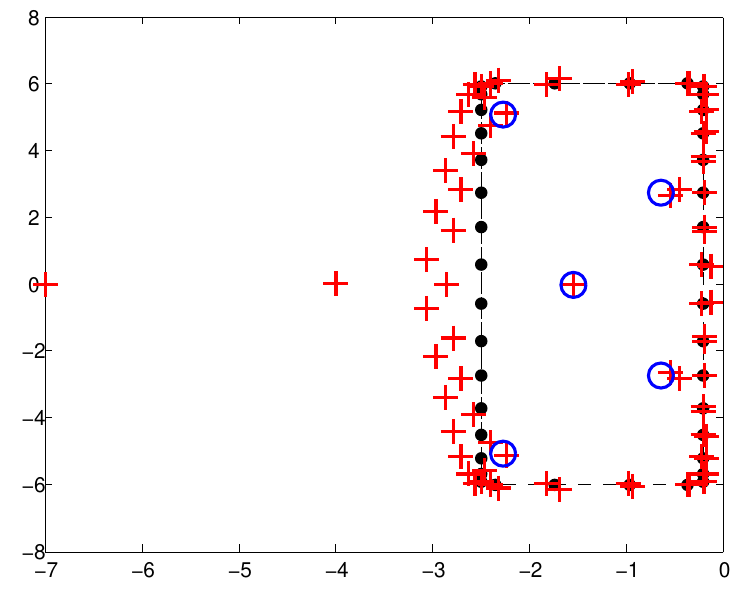}
	\caption{Left: Eigenvalues of \nref{eq:Kresex} obtained by solving the eigenvalue problem \nref{eq:sc} (plus) inside a
          rectangle defined by $c_2=-3-6\imath$ and $c_4=1+6\imath$
          with $m=40$ quadrature nodes (point) obtained 
          by Beyn's method (circle) 
          Right: Eigenvalues of \nref{eq:Kresex} obtained by solving the eigenvalue problem \nref{eq:sc} (plus) inside a
          rectangle defined by $c_2=-2.5-6\imath$ and $c_4=-0.2+6\imath$ with $m=40$ quadrature nodes (point) obtained by Beyn's method (circle). }
	\label{fig:KresRectangle}
      \end{figure}
      
\begin{figure}
\centering
	\includegraphics[width=0.45\textwidth]{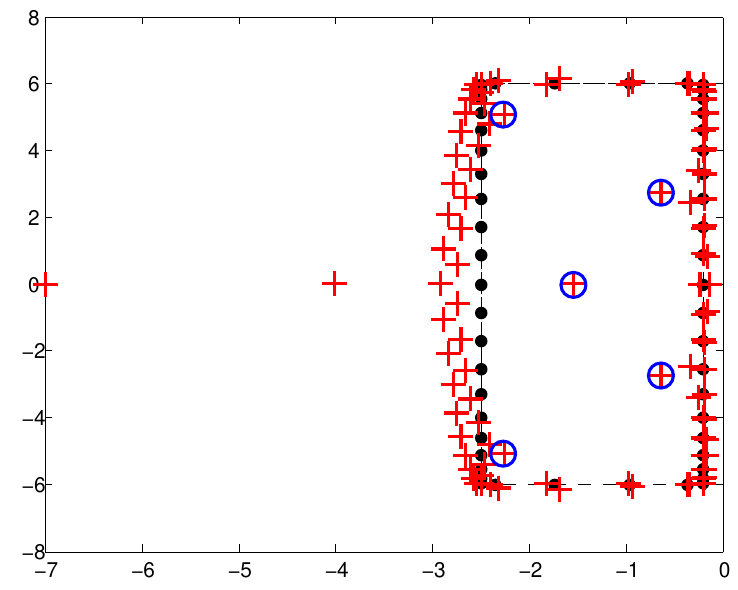}
\caption{Eigenvalues of \nref{eq:Kresex} obtained by solving the eigenvalue problem \nref{eq:sc} (plus) inside a
	rectangle defined by $c_2=-2.5-6\imath$ and $c_4=-0.2+6\imath$ with $m=50$ quadrature nodes (point) and by Beyn's method (circle).}
\label{fig:KresRectangleIncreasem}
\end{figure}

%
%
      
 \subsection*{Example 3: Hadeler problem}
 As an example of a general nonlinear eigenvalue problem,
 we  consider the \emph{Hadeler problem}~\cite{hadeler1967mehrparametrige,Ruh73,nlevp_collection}:
\begin{equation}\label{eqn:hadelerproblem}
T(z)=(e^z-1)B_1+z^2 B_2-B_0,
\end{equation}
with the coefficient matrices
\begin{equation}
B_0=b_0I,\ \ \ B_1=(b^{(1)}_{jk}),\ \ B_2=(b^{(2)}_{jk}),
\end{equation}
\begin{equation}
b^{(1)}_{jk}=(n+1-\max(j,k))jk,\ \ \ b^{(2)}_{jk}=n\delta_{jk}+1/(j+k),
\end{equation}
of dimension $n$ and a parameter $b_0=100$ 
(following reference~\cite{Ruh73}). For our experiments we choose $n=200$.
Note that the theoretical considerations in Section~\ref{sec:theory} assumed the matrix $A_0$ to be invertible,
this example  shows that our algorithm can be used for solving more general problems.

We refer the reader to the above articles for details on this problem.
The eigenvalues of \nref{eqn:hadelerproblem} are real with $n$ of them
being negative and $n$ positive. The eigenvalues become better spaced
as we move away from the origin and the smallest one is close to
$-48$.  We compute the eigenvalues inside a circle centered at $c =
-30$ with radius $r=11.5$.  We first determine the number of
quadrature nodes $m$ necessary to get a good rational approximations of
functions $f_1(z)=e^z-1$ and $f_2(z)=z^2$ inside the considered
circular contour $\Gamma$.  The right part of Figure \ref{fig:Hadler}
shows the approximation errors for the rational approximations of
$f_1(z)$ and $f_2(z)$ versus $m$. Based on Figure \ref{fig:Hadler} and
referring to proposition \ref{prop:1}, a degree of $m=32$ is
sufficient to approximate $T(z)$ up to the accuracy of $\mathrm{tol} =
10^{-12}$.  Using $m = 32$ Gauss-Legendre quadrature nodes, $12$
eigenvalues of \eqref{eqn:hadelerproblem} were computed using
shift-and-invert Arnoldi method applied to the large system
(\ref{eq:sc}--\ref{eq:sc1}).  These results are compared with the
approximations obtained by Beyn's first algorithm~\cite{Bey12}, see
left side of Figure \ref{fig:Hadler}.  For Beyn's method, $50$
quadrature points are required to reach a backward error of the $12$
eigenvalues smaller than $\delta = 10^{-10}$.

We repeat the same  experiment using  the reduced  subspace iteration
given by Algorithm \ref{alg:subsit}.
We first start with $\nu=40$ random vectors, where $\nu$ is the
dimension of the subspace in Algorithm \ref{alg:subsit} and then apply
$q=10$ steps of inverse iteration method, i.e.,
Algorithm~\ref{alg:shinv}, on each of these vectors separately to
obtain a block of $\nu$ vectors each one of size $n$.  Note that these
vectors are the resulting bottom parts of the final iterates of
Algorithm \ref{alg:shinv}. These $\nu$ bottom parts are orthogonalized
to obtain an orthonormal basis $U$ used to perform the Rayleigh-Ritz
projection that leads to a nonlinear eigenvalue problem in $\CC^{\nu}$
of the form \nref{eq:RR2}.  We then solve this reduced (nonlinear)
eigenvalue problem \nref{eq:RR2} by computing the eigenvalues and
eigenvectors of the corresponding expanded problem (\ref{eq:sc} --
\ref{eq:sc1}).  Note that this projected problem is now of size $(m+1) \nu \ll (m+1) n$.
Before each restart of Algorithm \ref{alg:subsit},
we select $\nu$ approximate eigenpairs whose eigenvalues are inside
the contour.  The initial vectors $w$ selected in line~3 of the
Algorithm \ref{alg:subsit}
are of the form $w =[v; \ u]$ where the components $v_i$ of $v$
satisfy $v_i = u /(\sigma_i -\lambda)$.  Here $(\lambda, u)$ is one of
the $\nu$ approximate eigenpairs computed from the previous outer
iteration. At the very first outer iteration $v$ and $u$ are random
vectors. The resulting bottom parts of the final iterates are used to
form the block $U$, see line~5 of Algorithm \ref{alg:subsit}.
The columns of $U$ are orthonormalized before applying the Rayleigh-Ritz
procedure in lines~6--7. At each level $\ell$ we use $2^{\ell-1}$
quadrature points. This multilevel approach allows us to obtain
several (coarse) approximations, one for each level $\ell$, using the
same set of $m=32$ (fine) quadrature points.
 Algorithm \ref{alg:subsit} computed the $12$ eigenvalues of
 interest requiring $L = 6$ outer iterations.  These results are
 compared with the approximations obtained by the AAA
 algorithm~\cite{NakST18} and the NLEIGS algorithm~\cite{GueVBMM14},
 see Figure \ref{fig:Hadler_Algo4_AAA}. For the AAA and NLEIGS
  algorithms, the boundary circle is discretized by $100$ equispaced
  points and an error tolerance of $10^{-12}$ is set for the rational
  interpolant. Note that $7$ and $38$ interpolation nodes are needed
  to approximate $T(z)$ up to the accuracy of $\mathrm{tol} =
  10^{-12}$ inside the circle, using the AAA and the NLEIGS algorithm,
  respectively. Finally the right side of
  Figure~\ref{fig:Hadler_Algo4_AAA} illustrates the nonlinear residual
  norm $\|T(\lambda) u\|_{\infty}$ for all $12$ eigenpairs $(\lambda,
  u)$ computed by Algorithm~\ref{alg:subsit}, and those computed by
  the AAA and the NLEIGS algorithm.

Another way to  extract those $12$ eigenvalues
of  interest is  to resort  to  the rational  approximation using  the
Cauchy integral  formula inside an elliptic contour.  We  consider an
ellipse centered  at $c=-30$ with  semi-major axis $r_x=10$ and semi-minor axis $r_y=1$  on the
$x$-axis  and  $y$-axis,  respectively.
We can solve the expanded problem (\ref{eq:sc} -- \ref{eq:sc1}) with
$m=8$ and $\sigma=c$ and perform as many steps as needed to extract
all $12$ real eigenvalues inside the elliptic contour.  In Figure
\ref{fig:HadelerEllipse}, we present the eigenvalues obtained solving
the expanded problem (\ref{eq:sc} -- \ref{eq:sc1}) using Algorithm~
\ref{alg:subsit} and Beyn's integral method with the same elliptic
contour.  For Beyn's method $20$ trapezoidal quadrature nodes are
necessary to compute all $12$ eigenvalues of interest with the
backward error smaller than $\delta = 10^{-10}$.

\begin{figure}
	\centering
  \includegraphics[width=0.45\textwidth]{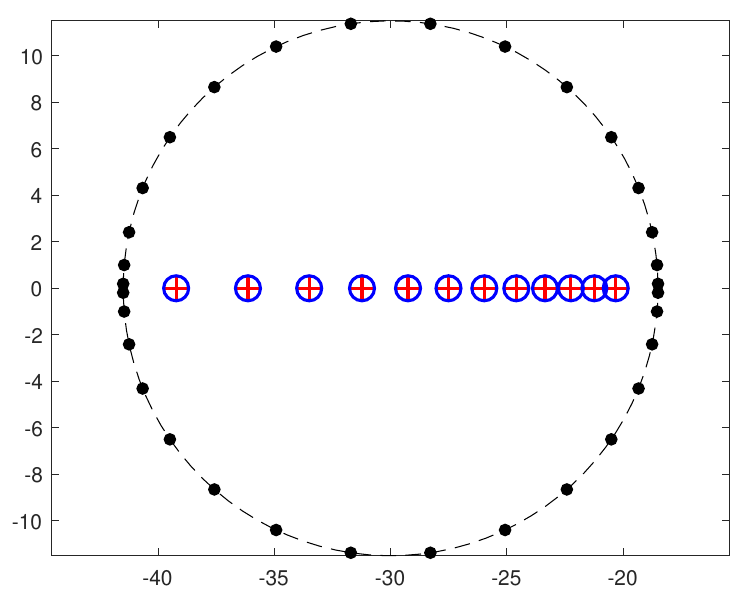}
  \includegraphics[width=0.45\textwidth]{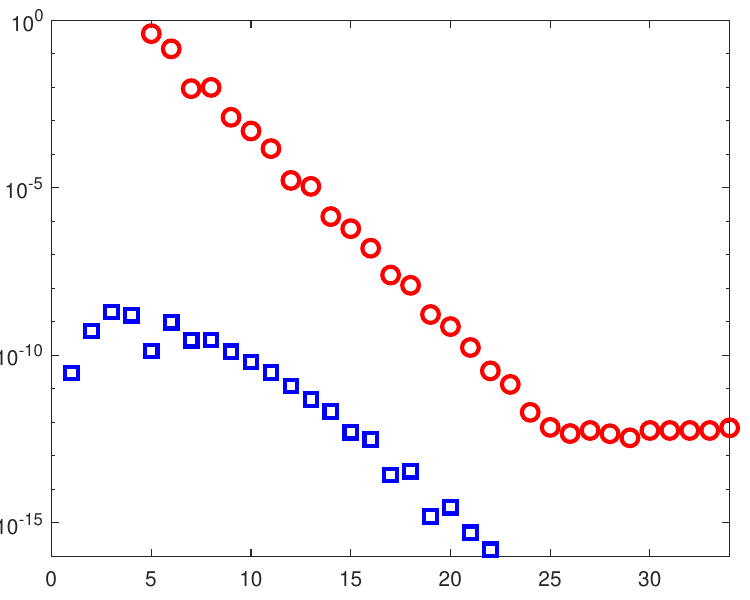}
  \caption{Left:  Eigenvalues of  (\ref{eqn:hadelerproblem}) inside  a
    circle of radius  $r=11.5$ and center $c=-30$  obtained by computing the  eigenvalues of the expanded problem (\ref{eq:sc} --  \ref{eq:sc1})  (plus) and by  Beyn's method  (circle).  Right:  The errors $e_m$ of the rational approximation of $e^{-z}$ (square) and $z^2$ (circle)  versus  the  number   of  quadrature  nodes
    $m$.}\label{fig:Hadler}
\end{figure}
\begin{figure}
	\centering
	\includegraphics[width=0.48\textwidth]{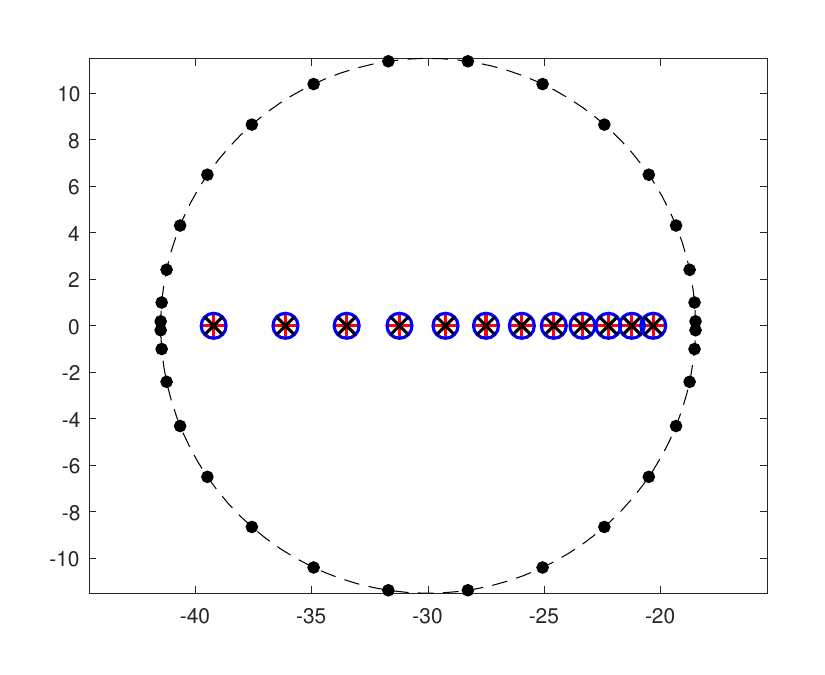}
	\includegraphics[width=0.48\textwidth]{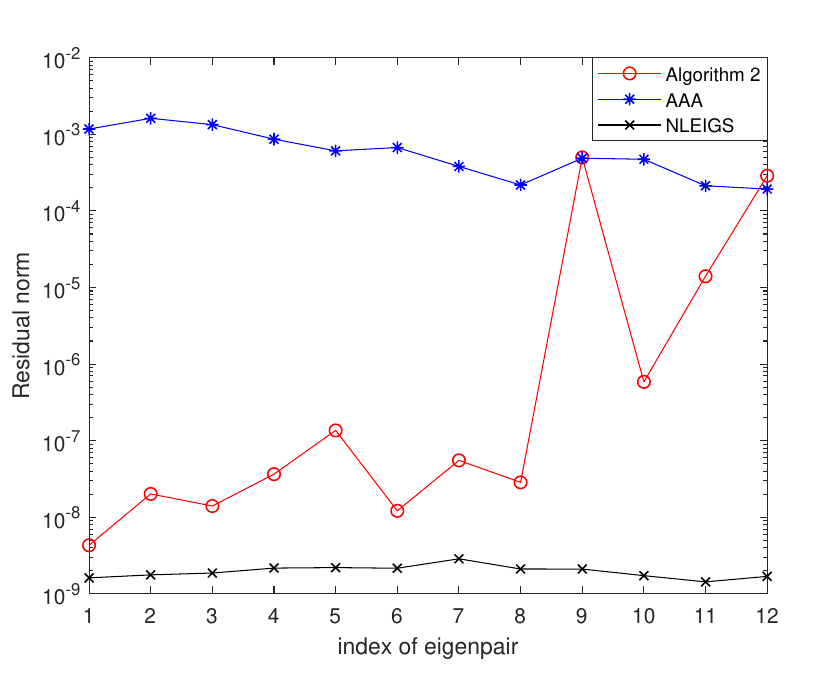}
	\caption{Left:  Eigenvalues of  (\ref{eqn:hadelerproblem}) inside  a
		circle of radius  $r=11.5$ and center $c=-30$  obtained by Algorithm
		\ref{alg:subsit}  (plus), the AAA algorithm  (circle) and the NLEIGS algorithm (cross).  Right:  The
		residual norm $\|T(\lambda) u\|_{\infty}$ of the computed eigenpairs.}\label{fig:Hadler_Algo4_AAA}
\end{figure}

\begin{figure}
	\centering
	\begin{minipage}[c]{.60\linewidth}
		\includegraphics[width=7cm]{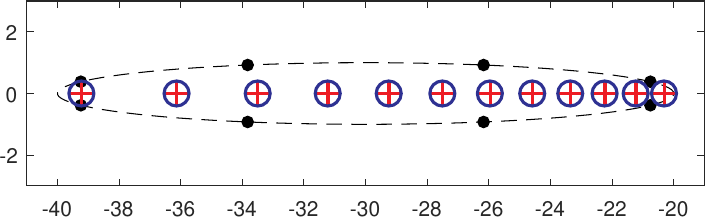}
	\end{minipage}
	\caption{Eigenvalues  of (\ref{eqn:hadelerproblem})  inside an
          ellipse  centered  at  $c=-30$   and  with  semi-major  axis
          $r_x=10$ and  semi-minor axis $r_y=1$ obtained  by solving the
          expanded problem (\ref{eq:sc} --  \ref{eq:sc1})  (plus) and by Beyn's method (circle).}
	\label{fig:HadelerEllipse}
\end{figure}

\section{Concluding remarks} \label{sec:concl} An appealing feature of
the  general approach  proposed in  this paper  for solving  nonlinear
eigenvalue problems is its simplicity.  A general nonlinear problem is
approximated  by   a  rational   eigenvalue  problem  which   is  then
linearized.  The  resulting  linear  problem provides  the  basis  for
developing a number of methods and one of them, among possibly many
others, is discussed in this paper.  The theory allows to exactly
predict which eigenvalues of the original problem  are  well approximated
and to ensure that no eigenvalues in the region will be missed.
Another attribute of the proposed method is its flexibility. It is possible to
compute  all  eigenvalues  in  a  union of    small  regions  each
requiring a small  number of poles, or to use one single large region
to compute  many eigenvalues but now with a large number of
poles. The question as to how to optimally exploit these trade-offs remains 
to be  further investigated.
Finally, the method has a good potential for solving realistic
large sparse nonlinear eigenvalue problems such as those mentioned in
the introduction. For more details, we refer the reader to~\cite{ElGuide19}.
In this regard, we note that \emph{only one factorization
is required} namely that of $S(\sigma)$ where $\sigma$ is the shift
in the shift-and-invert procedure. In many applications the patterns
of the matrices $A_j$ are not too different from one another and so 
$S(\sigma)$, which is itself a combination of the $A_j$s, will remain sparse.

\bibliographystyle{siam} 
\bibliography{SaaEM20} 

\newpage

\appendix
\section{Additional Numerical Experiments}

\subsection*{Example 4}
We consider the following nonlinear eigenvalue problem, see~\cite{Sol06,Kre09},
	\eq{eq:FEex}
	T(z)=B_0+zA_0+\frac{1}{1-z}e_ne_n^T,
	\en
	with
	$$B_0=n\begin{pmatrix} 
	2 & -1 & &&\\
	-1 &  \ddots &  \ddots&\\
	&  \ddots & 2&-1\\
	&&-1&1
	\end{pmatrix},\quad A_0=-\frac{1}{6n}\begin{pmatrix} 
	4 & 1 & &&\\
	1 &  \ddots &  \ddots&\\
	&  \ddots & 4&1\\
	&&1&2
	\end{pmatrix},$$
resulting from the finite element discretization of the nonlinear boundary eigenvalue problem
\eq{eq:nbep}
-u''(x)=\lambda u(x), 0\leq x\leq1,\quad u(0)=u'(1)+\frac{\lambda}{\lambda-1}u(1)=0.
\en

To   compare    our   results   with   those    obtained   by   Beyn's
method~\cite[Example  4.11]{Bey12}, we  consider the case when  $n=100$ and  compute
five eigenvalues enclosed by a  circle centered at $c=150$ with radius
$r=150$.  We first  determine the  number of  quadrature
nodes  $m$ needed  to get  a  good  rational  approximation of the function
$f(z)=\frac{1}{1-z}$   inside   the    considered   circular   contour
$\Gamma$.  The  right part  of  Figure  \ref{fig:RatAppEx2} shows  the
approximation error for  the rational approximation of $f(z)$ versus $m$.

Because $f(z)$ is itself a rational function 
  a high  enough accuracy  is
  obtained for a small value of $m$, namely $m=6$.
  Therefore, we solve the
  expanded problem (\ref{eq:sc} --  \ref{eq:sc1}) with $m=6$ and
$\sigma=c$  to get the  approximate  eigenvalues inside  the circle.   The
computed eigenvalues    are     shown    on     the    left     of
Figure~\ref{fig:RatAppEx2}. These  results are directly  compared with
the ones obtained by Beyn's first algorithm with the backward error of the $5$
eigenvalues being smaller than $\delta = 10^{-4}$. To obtain this level of accuracy,
$50$ quadrature nodes are needed to determine the $5$ eigenvalues inside the circle.

Note that the function $f(z)$ is already  given in a rational form
and so we  can  solve \nref{eq:FEex}  by considering
the same linearization as the one invoked in Section~\ref{sec:alg}:
\eq{eq:FEexLinea}
\begin{bmatrix} 
	I & -I \\
	e_ne_n^T & B_0
\end{bmatrix} \begin{bmatrix} \frac{u}{1-z} \\ u \end{bmatrix} =z \begin{bmatrix} 
I & 0\\
0 & -A_0
\end{bmatrix} \begin{bmatrix} \frac{u}{1-z} \\ u \end{bmatrix}.
\en 

\noindent
Figure \ref{fig:FexLinea} compares the eigenvalues obtained by solving \nref{eq:FEexLinea} and the ones obtained by  computing the  eigenvalues and of the
expanded problem (\ref{eq:sc} --  \ref{eq:sc1}).


Alternatively,  we  can  solve problem~\nref{eq:FEex} directly   using
Algorithm~\ref{alg:subsit} without restarts, i.e., with only one outer loop
($\ell =1$ only). Let $w_i = [ v_i; \ u_i ]$, $i=1,...,\nu$ be a set of
$\nu$ random vectors  of  size $N  =  (m+1)n$,  where $v_{i}\in\mathbb{R}^{mn}$  and
$u_{i}\in\mathbb{R}^{n}$.  For  this experiment, we  choose $\nu=7$ and we  apply
$q=5$ steps of  Algorithm  \ref{alg:shinv}  on each
$w_{i}$.  We orthogonalize
the  resulted vectors  $U  = [u_1,u_2,\ldots,u_\nu]$  to  obtain a  good
subspace  to perform  the  projection method  introduced in  Algorithm
\ref{alg:subsit}.

We set $m=6$ and invoke one outer iteration of  Algorithm \ref{alg:subsit} to
compute the eigenvalues enclosed by the same circular contour.  As can
be seen on the left of Figure \ref{fig:FexLinea}, the eigenvalues
obtained are in good agreement with
those obtained by computing the eigenvalues of the expanded problem
(\ref{eq:sc} -- \ref{eq:sc1}) and by Beyn's method.
\begin{figure}
	\centering
		\includegraphics[width=0.45\textwidth]{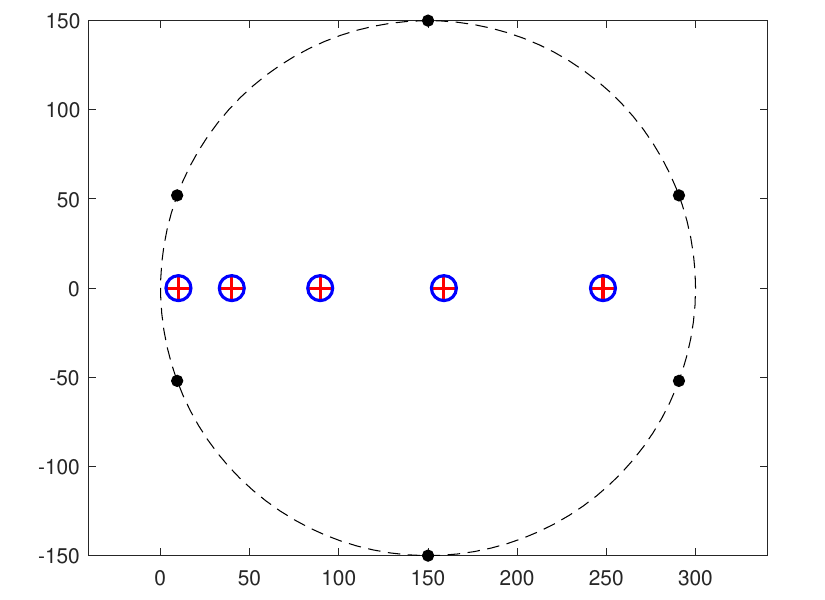}
		\includegraphics[width=0.45\textwidth]{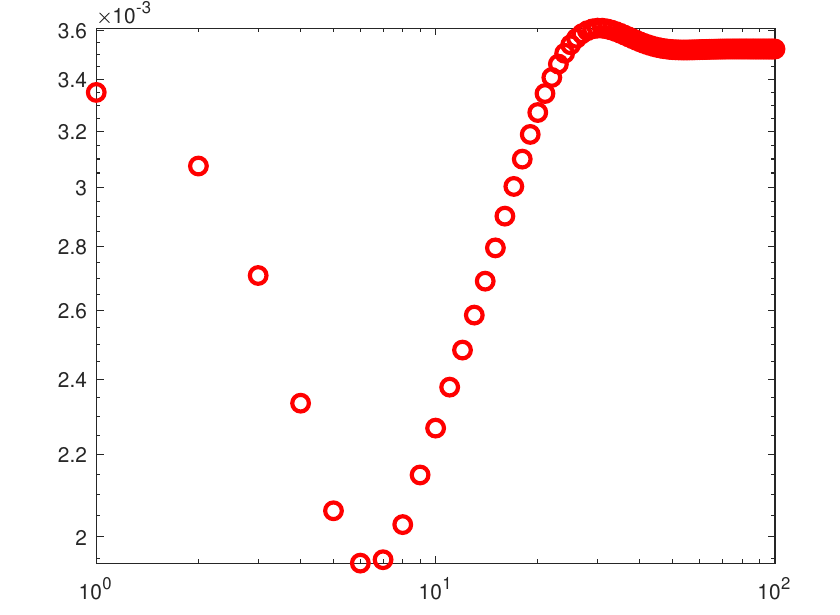}
	\caption{Left: Eigenvalues of \nref{eq:FEex} inside a
		circle of radius $r=150$ centered at $c=150$ obtained by computing the  eigenvalues of the
		expanded problem (\ref{eq:sc} --  \ref{eq:sc1}) (plus) and  
		by Beyn's method (circle).
		Right: The errors $e_m$ of the rational approximation of $f(z) = \frac{1}{1-z}$ versus the number of quadrature nodes $m$. }
		\label{fig:RatAppEx2}
\end{figure}
\begin{figure}
	\centering
	\includegraphics[width=0.45\textwidth]{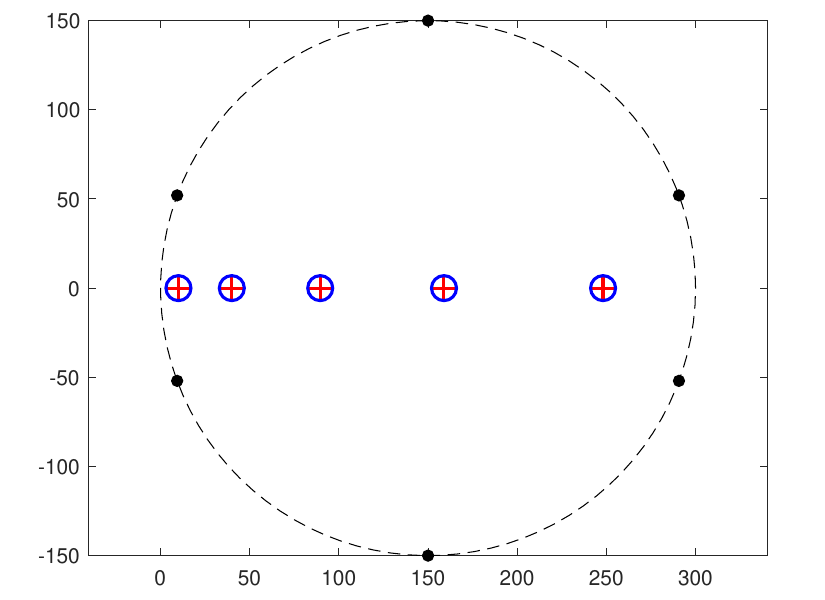}
	\includegraphics[width=0.45\textwidth]{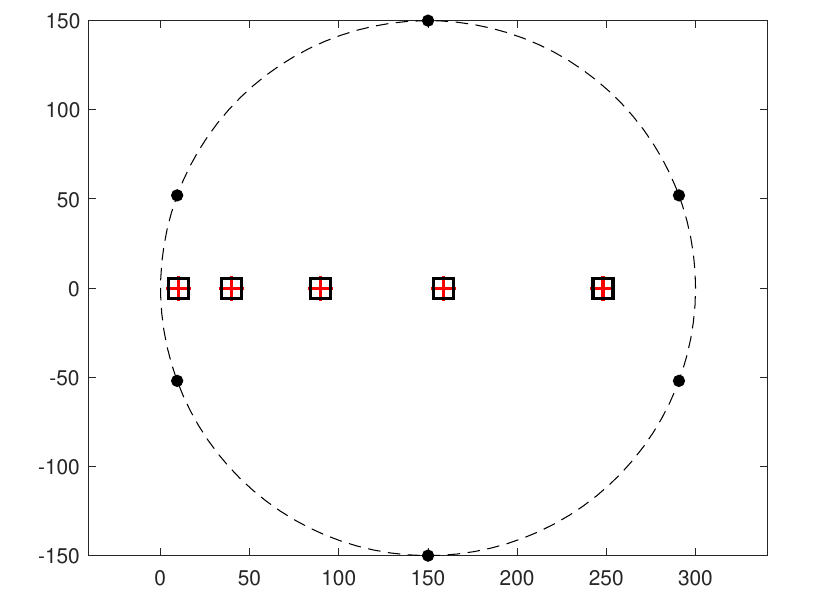}
	\caption{Left: Eigenvalues of \nref{eq:FEex} inside a
		circle of radius $r=150$ centered at $c=150$ obtained by Algorithm \ref{alg:subsit} (plus) and
		by Beyn's method (circle). Right: Eigenvalues of \nref{eq:FEex} inside a
		circle of radius $r=150$ centered at $c=150$ obtained by computing the  eigenvalues of the
		expanded problem (\ref{eq:sc} --  \ref{eq:sc1}) (plus) and
		by linearization \eqref{eq:FEexLinea} (square).}
	\label{fig:FexLinea}
      \end{figure}

%
%
\subsection*{Example 5: Butterfly Problem}\label{exp:butterfly}
To illustrate the behavior of rational approximation methods
when using a contour centered
at an  arbitrary point in  the complex  plane, we present a few results 
with  the \emph{butterfly problem} (so called 
because of the  distribution of its eigenvalues in  the complex plane)
available from the  NLEVP   collection~\cite{nlevp_collection}.
This is a quartic eigenvalue problem of the form
\begin{equation}
\label{eq:butterfly}
T(\lambda) = A_0 + \lambda A_1 + \lambda^2 A_2 + \lambda^3 A_3 +
\lambda^4 A_4 ,
\end{equation}
where   $A_0, A_1, \ldots, A_4$
are  structured matrices of  size  $n=64$.  The  $256$
eigenvalues  of   this  problem  are   shown  on  the  left side of 
Figure~\ref{fig:butterfly}. A detailed  description  of this  example can
be found  in \cite{MehW02}.  We compute the  eigenvalues and vectors of the
expanded problem (\ref{eq:sc} --  \ref{eq:sc1})  with
$m=50$  quadrature nodes  to  compute the  eigenvalues  enclosed by  a
circle  centered  at $c=1+1\imath$  with  radius  $r=0.5$. We  compare
approximations of  $13$ computed eigenvalues with  those determined by
the linearization  of problem \eqref{eq:butterfly} and  by application
of  Beyn's  method.   The  right  part  of  Figure~\ref{fig:butterfly}
summarizes  our findings.   Alternatively,  we can  use a  rectangular
contour    as    described    already   in    Example    2.     Figure
\ref{fig:butterfly_reg}  shows  approximations   of  $17$  eigenvalues
enclosed     by    the     rectangular     contour  defined by the     corners
$c_2=(0.55,0.48)$ and $c_4 = (1.2,1.3)$
obtained by solving the
expanded problem (\ref{eq:sc} --  \ref{eq:sc1})  with
$m=90$ quadrature nodes and direct linearization.
\begin{figure}
	\centering
  \includegraphics[height=4.7cm, width=0.45\textwidth]{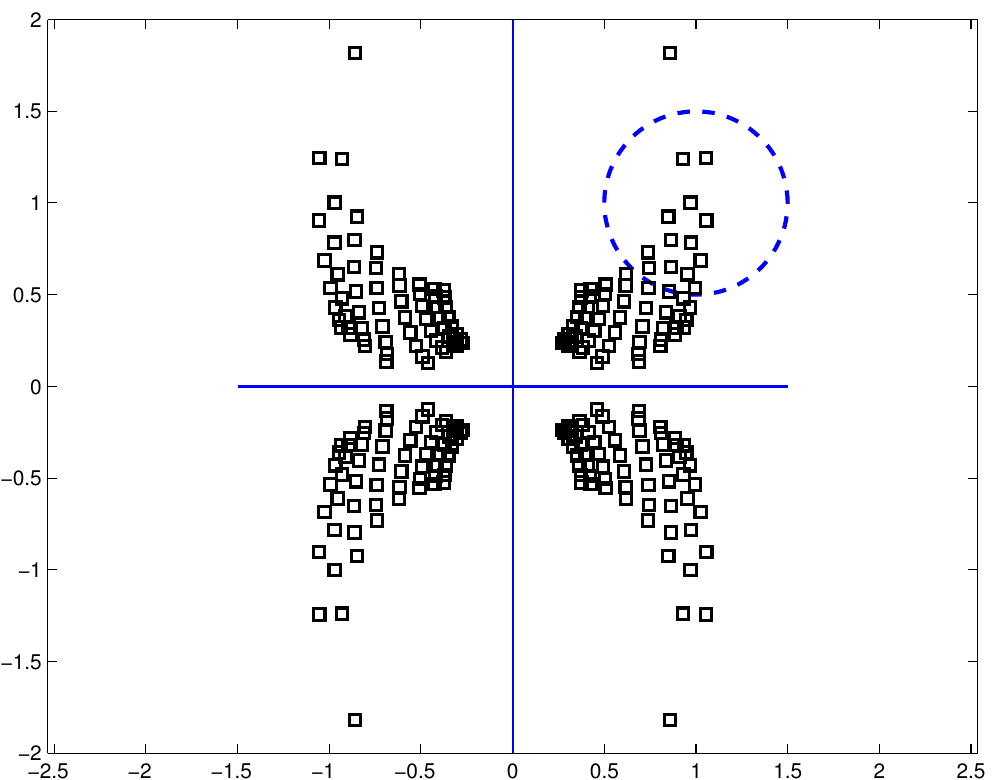}
  \includegraphics[height=4.7cm, trim={0 10 0 1}, width=0.45\textwidth]{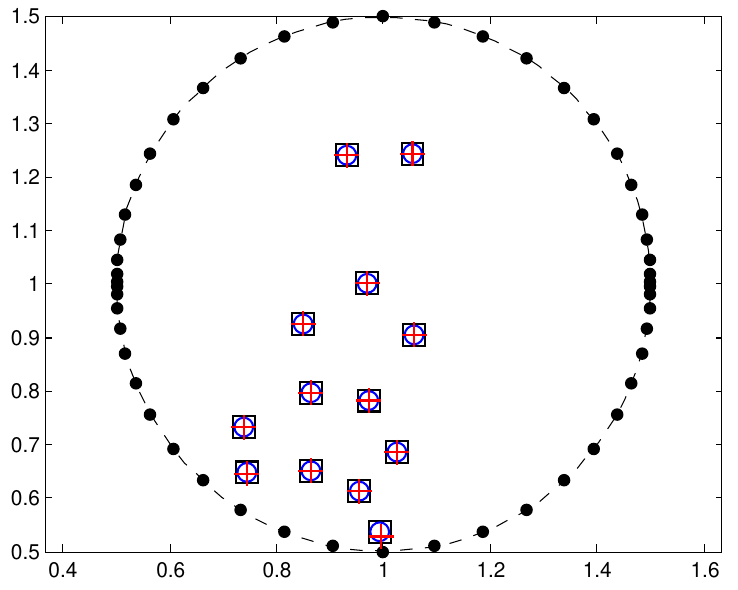}
  \caption{Left: All $256$ eigenvalues of butterfly example \eqref{eq:butterfly} (square) obtained by linearization. Circle contour centered at $c=1+1\imath$ with radius $r=0.5$ (dashed). Right: Eigenvalues of butterfly example \eqref{eq:butterfly} inside a
		circle of radius $r=0.5$ centered at $c=1+1\imath$ obtained by computing the  eigenvalues  of the
		expanded problem (\ref{eq:sc} --  \ref{eq:sc1}) (plus) with $m=50$, by linearization (square) and by Beyn's method (circle).}\label{fig:butterfly}
\end{figure}
\begin{figure}
	\centering
  \includegraphics[height=4.7cm,width=0.45\textwidth]{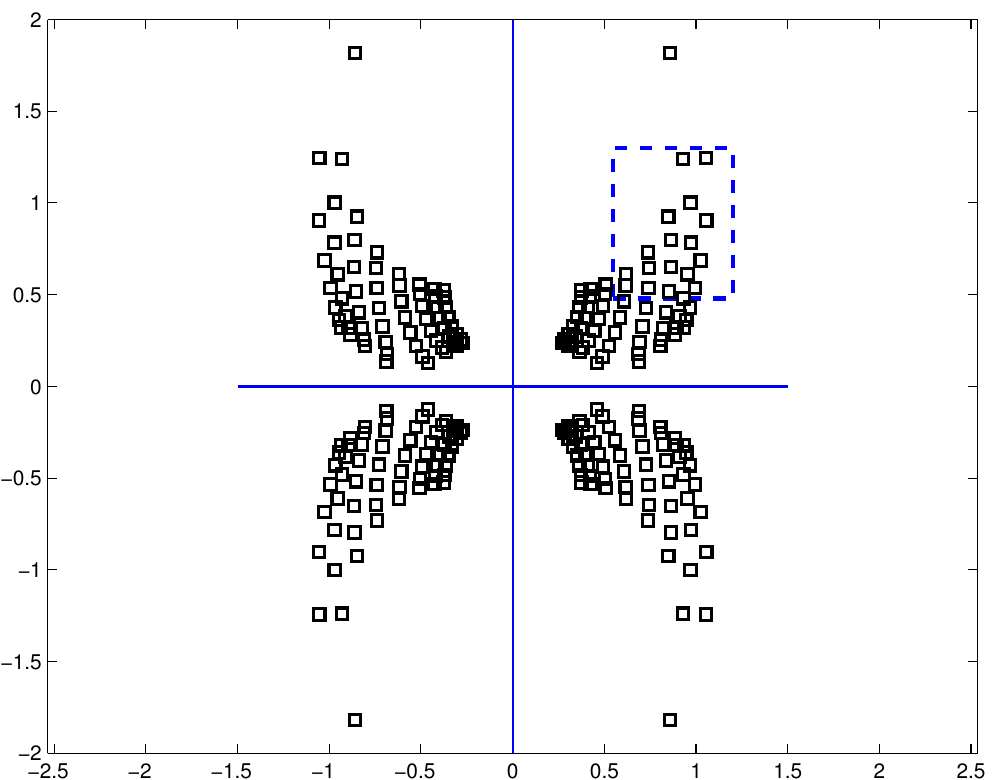}
  \includegraphics[height=4.7cm,trim={0 10 0 1},width=0.45\textwidth]{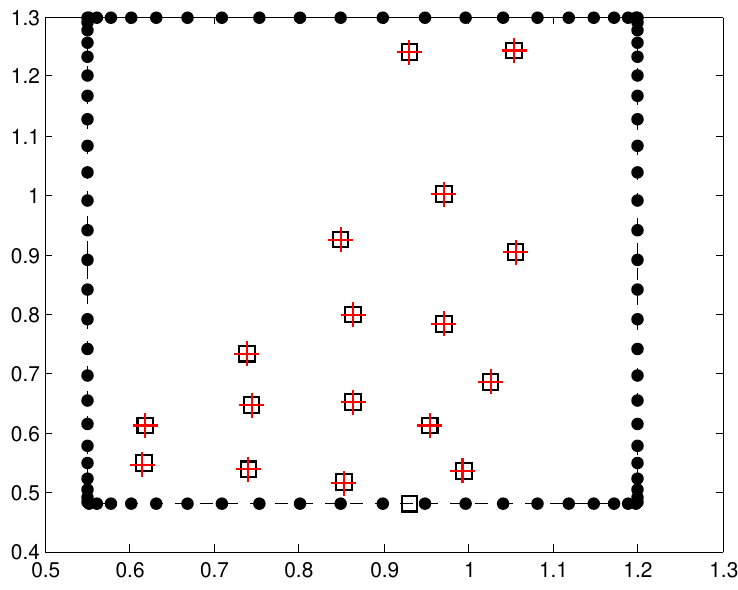}
  \caption{Left:   All   $256$   eigenvalues  of   butterfly   example
    \eqref{eq:butterfly}          (square)         obtained          by
    linearization.      Rectangular      contour      centered      at
    $0.875         +          0.89\imath$         with         corners
    $(0.55,1.3),    (0.55,0.48),     (1.2,0.48)$    and    $(1.2,1.3)$
    (dashed).    Right:     Eigenvalues    of     butterfly    example
    \eqref{eq:butterfly}  inside  a  rectangular contour  obtained  by computing the  eigenvalues  of the
    expanded problem (\ref{eq:sc} --  \ref{eq:sc1})   (plus)    with    $m=90$   and    by
    linearization.}\label{fig:butterfly_reg}
\end{figure}

\end{document}